\documentclass{article}

\usepackage{amsfonts}
\usepackage{amsmath}
\usepackage{theorem}
\usepackage{array}
\usepackage[a4paper]{geometry}
\usepackage{amssymb}
\usepackage{graphics}
\usepackage[usenames]{color}
\usepackage[all]{xypic}
\usepackage{graphicx}
\usepackage{boxedminipage}

\newtheorem{theorem}{Theorem}[section]

\newtheorem{corollary}[theorem]{Corollary}

\newtheorem{definition}[theorem]{Definition}
\newtheorem{remark}[theorem]{Remark}

\newtheorem{lemma}[theorem]{Lemma}

\newtheorem{proposition}[theorem]{Proposition}

\newenvironment{proof}[1][Proof]{\textbf{#1.} }{\ \rule{0.5em}{0.5em}}

\newcommand{\diff}{\mathrm{Diff}}
\newcommand{\di}{\mathrm{div}}
\newcommand{\fp}{\mathbb{F}_p}
\newcommand{\fq}{\mathbb{F}_q}
\newcommand{\fqa}{\mathbb{F}_{q^a}}
\newcommand{\fqj}{\mathbb{F}_{q^j}}
\newcommand{\fqk}{\mathbb{F}_{q^k}}
\newcommand{\fl}{\mathbb{F}_\ell}
\newcommand{\flc}{\overline{\mathbb{F}}_\ell}
\newcommand{\lcm}{\mathrm{lcm}}
\newcommand{\tr}{\mathrm{Tr}}
\newcommand{\z}{\mathbb{Z}}

\begin{document}
\title{Towers of Function Fields over Non-prime Finite Fields}
\author{Alp Bassa\footnote{Alp Bassa is supported by T\"{u}bitak Proj. No. 112T233. Part of this work was done while Alp Bassa was at CWI, Amsterdam (Netherlands).}, Peter Beelen\footnote{Peter Beelen gratefully acknowledges the support from the Danish National Research Foundation and the National Science Foundation of China (Grant No.11061130539) for the Danish-Chinese Center for Applications of Algebraic Geometry in Coding Theory and Cryptography.}, Arnaldo Garcia\footnote{Arnaldo Garcia was partially supported by CNPq (Brazil), Sabanc\i \  University and T\"{u}bitak (Turkey).}, and Henning Stichtenoth\footnote{Henning Stichtenoth was partially supported by T\"{u}bitak Proj. No. 111T234.}}
\date{\empty}

\maketitle

\begin{abstract} Over all non-prime finite fields, we construct some recursive towers of function fields with many rational places. Thus we obtain a substantial improvement on all known lower bounds for Ihara's quantity $A(\ell)$, for $\ell = p^n$ with $p$ prime and  $n>3$ odd. We relate the explicit equations to Drinfeld modular varieties.
\end{abstract}

\section{Introduction}\label{sec:one}

Investigating the number of points on an algebraic curve over a finite field is a classical
subject in Number Theory and Algebraic Geometry. The origins  go back to Fermat, Euler and
Gauss, among many others. The basic result is A. Weil's theorem which is equivalent
to the validity of Riemann's Hypothesis in this context. New impulses  came from Goppa's
construction of good codes from curves with many rational points, and also from applications
to cryptography. For information, we refer to \cite{gee,niexin}.

\vspace{.2cm} \noindent
One of the main open problems in this area of research is the determination of Ihara's quantity $A(\ell)$
for non-square finite fields; i.e., for cardinalities $\ell = p^n$ with $p$
prime and $n$ odd.  This quantity controls the asymptotic behaviour of the number of $\fl$-rational points
(places) on algebraic curves (function fields) as the genus increases. This is the topic of our paper.

\vspace{.2cm} \noindent
Let $F$ be an algebraic function field of one variable over $\fl$, with the field $\fl$  being algebraically
closed in $F$. We denote
$$N(F)=\text{number of }\mathbb F_\ell \text{-rational places of }F, \ \makebox{ and } \
g(F)=\text{genus of }F.$$
The  Hasse--Weil upper bound states that
$$N(F)\leq 1+\ell+2\sqrt{\ell} \cdot g(F).$$
This upper bound was improved by Serre \cite{ser1} who showed that the factor $2\sqrt{\ell}$ can
be replaced above by its integer part $\lfloor 2\sqrt{\ell} \rfloor$.

\vspace{.2cm} \noindent
Ihara \cite{iha2} was the first to realize that the Hasse--Weil upper bound becomes weak when the genus $g(F)$ is large
with respect to the size $\ell$ of the ground field $\fl$. He introduced the  quantity

\begin{equation} \label{al} A(\ell)=\limsup_{g(F)\to \infty} \frac{N(F)}{g(F)}, \end{equation}
where the limit is taken over all function fields $F/\fl$ of genus $g(F)>0$. By Hasse--Weil it holds that $A(\ell) \le
2 \, \sqrt{\ell}$, and Ihara showed that $A(\ell) <  \sqrt{2\ell}$.

\vspace{.2cm} \noindent
The best upper bound known is due to Drinfeld--Vl\u{a}du\c{t} \cite{drivla}. It states that
\begin{equation}\label{dv}
A(\ell)\leq \sqrt{\ell}-1, \text{ for any prime power } \ell .\end{equation}
If $\ell$ is a square, the opposite inequality $A(\ell) \ge \sqrt{\ell} -1$ had been shown earlier by
Ihara using the theory of modular curves (see \cite{iha1}). Hence
\begin{equation}\label{dvequality}
A(\ell) = \sqrt{\ell} -1 \, , \ \makebox{ when $\ell$ is a square.}
\end{equation}
\noindent
For all other cases when the cardinality $\ell$ is a non-square, the exact value of the quantity $A(\ell)$ is not known.
Tsfasman--Vl\u{a}du\c{t}--Zink \cite{tvz} used Equation\,(\ref{dvequality}) to prove the existence of long linear codes with relative parameters above the Gilbert--Varshamov bound, for finite fields of square cardinality $\ell = q^2 $ with $q \ge 7$. They also gave a proof of Equation\,(\ref{dvequality}) in the cases $\ell=p^2$ or $\ell=p^4$ with $p$ a prime number.

\vspace{.2cm} \noindent
To investigate  $A(\ell)$ one introduces the notion of (infinite) towers of $\mathbb F_\ell$-function fields:
$$\mathcal F=(F_1 \subseteq F_2 \subseteq F_3\subseteq \ldots \subseteq F_i\subseteq \ldots),$$
where all $F_i$ are function fields over $\fl$, with $\fl$ algebraically closed in $F_i$, and $g(F_i) \to \infty$
as $i \to \infty$. Without loss of generality one can assume that all extensions $F_{i+1}/F_i$ are  separable. As follows from Hurwitz' genus formula, the limit below exists (see \cite{jnt}) and it is called the limit
of the tower $\mathcal F$:
$$\lambda(\mathcal F):=\lim_{i\to \infty} \frac{N(F_i)}{g(F_i)}.$$

\noindent Clearly, the limit of any tower $\mathcal F$ over $\fl$ provides a lower bound for $A(\ell)$; i.e.,
$$ 0 \le \lambda (\mathcal F) \le A(\ell), \ \makebox{ for any $\fl$-tower } \mathcal F .
$$
So one looks for towers with big limits in order to get good lower bounds for Ihara's quantity.
Serre \cite{ser1} used Hilbert classfield towers to show that for all prime powers $\ell$,
\begin{equation} A(\ell) > c \cdot \log_2(\ell) \ , \ \makebox{ with an absolute constant $c > 0$ }.
\end{equation}
One can take $c = 1/96$, see \cite[Theorem 5.2.9]{niexin}. When $\ell = q^3$ is a cubic power, one has
the lower bound
\begin{equation}\label{aq3} A(q^3)\geq \frac{2(q^2-1)}{q+2}, \text{ for any prime power }q.\end{equation}
When $q=p$ is a prime number, this  bound was obtained by Zink \cite{zin}  using degenerations of modular surfaces.
The proof of Equation\,(\ref{aq3}) for general $q$ was given by Bezerra, Garcia and Stichtenoth \cite{bezgarsti}
using recursive towers of function fields; i.e.,  towers which are given in a recursive way by explicit
 polynomial equations. The concept of recursive towers turned out to be very fruitful for constructing towers with a large limit.

\vspace{.2cm} \noindent
An $\fl$-tower $\mathcal F = (F_1 \subseteq F_2 \subseteq F_3 \subseteq \ldots )$
 is recursively defined by $f(X,Y)\in \mathbb F_\ell[X,Y]$ when
\begin{itemize}
\item [(i)]$F_1=\mathbb F_\ell (x_1)$ is the rational function field, and
\item [(ii)] $F_{i+1}=F_i(x_{i+1})$ with $f(x_i,x_{i+1})=0$, for all $i\geq 1$.
\end{itemize}

\noindent
For instance, when $\ell=q^2$ is a square, the polynomial (see \cite{jnt})
$$f(X,Y)=(1+X^{q-1})(Y^q+Y)-X^q \in \mathbb F_{q^2}[X,Y]$$
defines a recursive tower over $\mathbb F_{q^2}$ whose limit $\lambda (\mathcal F )=q-1$ attains
 the Drinfeld--Vl\u{a}du\c{t} bound.

\vspace{.2cm} \noindent
When $\ell=q^3$ is a cubic power one can choose the polynomial (see \cite{bezgarsti})
\begin{equation}\label{bgs} f(X,Y)=Y^q(X^q+X-1)-X(1-Y)\in \mathbb F_{q^3}[X,Y] \end{equation}
to obtain a recursive tower over $\mathbb F_{q^3}$ with limit
$
\lambda (\mathcal F) \ge {2(q^2-1)}/({q+2}) \, ;
$
this is the proof of Equation\,(\ref{aq3}) above. The  case $q=2$
of Equation\,(\ref{bgs}) is due to van der Geer--van der Vlugt \cite{geevlu}.

\vspace{.2cm} \noindent
In the particular case of a prime field, no explicit or modular tower with positive
limit is known; only variations of Serre's classfield tower method have been successful in this case, see
\cite{duumak,niexin}.

\vspace{.2cm} \noindent
All known lower bounds for $A(p^n)$, with a prime number $p$ and an odd exponent $n>3$, are rather weak, see
\cite{li,niexin}. For example, one has for $q$ odd and $n \ge 3$ prime (see \cite{lm})
\begin{equation}
A(q^n) \ge \frac{4q+4}{\lfloor \frac{3+\lfloor 2\, \sqrt{2q+2}\rfloor}{n-2}\rfloor +\lfloor 2\, \sqrt{2q+3}\rfloor} \, .
\end{equation}

\vspace{.2cm} \noindent
The main contribution of this paper is a new lower bound on $A(p^n)$ that gives a substantial improvement over all
known lower bounds, for any prime $p$ and any odd $n>3$. For large $n$ and small $p$, this new bound is rather
close to the Drinfeld--Vl\u{a}du\c{t} upper bound for $A(p^n)$. Moreover, it is obtained through a
recursive tower with an explicit polynomial $f(X,Y) \in \fp[X,Y]$.

\vspace{.2cm} \noindent
Our lower bound is:

\begin{theorem}\label{thm1}
Let $p$ be a prime number and $n=2m+1\ge 3$ an odd integer. Then
$$
A(p^n) \ge \frac{2(p^{m+1}-1)}{p+1+\epsilon} \ \makebox{ with } \ \epsilon = \frac{p-1}{p^m-1} \, .
$$
\end{theorem}
\noindent
In particular $A(p^{2m+1})>p^m-1$, which shows that a conjecture by Manin \cite{geelin,man}
is false for all odd integers $n=2m+1\ge 3$.

\vspace{.2cm} \noindent
The bound of Drinfeld--Vl\u{a}du\c{}t can be written as
$$
A(p^{2m+1}) \le p^m \cdot \sqrt{p} -1\, .
$$
Fixing the prime number $p$, we get
$$
\frac{\makebox{our lower bound}}{\makebox{Drinfeld--Vl\u{a}du\c{t} bound}} \to \frac{2 \, \sqrt{p}}{p+1}\, , \
\makebox{ as } \ m \to \infty \, .
$$
For $p=2$ we have $2\, \sqrt{p}/(p+1) \approx 0.9428\ldots$, hence our lower bound is only around
$6 \, \%$ below the Drinfeld--Vl\u{a}du\c{t} upper bound, for large odd-degree extensions of the binary
field $\mathbb{F}_2$.

%

\vspace{.2cm} \noindent
Now we give the defining equations for the several recursive towers $\mathcal F$ that we consider in this paper. Let
$\fl$ be a non-prime field and write $\ell = q^n$ with $n \ge 2$. Here the integer $n$ can be even or odd. For every partition of $n$ in relatively prime parts,
\begin{equation}\label{partition}
n=j+k \ \makebox{ with } \ j \ge 1, \, k \ge 1 \ \makebox{ and } \ \gcd (j,k)=1 \,,
\end{equation}
we consider the recursive tower $\mathcal F$ over $\fl$ that is given by the equation
\begin{equation}\label{definingeq}
{\rm Tr}_j\biggl(\dfrac{Y}{X^{q^k}}\biggr) +{\rm Tr}_k\biggl(\dfrac{Y^{q^j}}{X}\biggr) = 1 \, ,
\end{equation}
where
$$
{\rm Tr}_a(T) := T + T^q + T^{q^2} + \cdots +T^{q^{a-1}} \ \makebox{ for any }\ a \in \mathbb {N} \ .
$$

\begin{theorem}\label{thm3} \  Equation\,{\rm ({\rm \ref{definingeq}})} defines a recursive tower $\mathcal F$
over $\fl$ whose limit satisfies
$$
\lambda (\mathcal F) \, \ge \, 2\, \biggl(\frac{1}{q^j-1}+ \frac{1}{q^k-1} \biggr)^{-1} ;
$$
i.e., the harmonic mean of $q^j-1$ and $q^k-1$ is a lower bound for $\lambda (\mathcal F)$.
\end{theorem}
\noindent
The very special case $n=2$ and $k=j=1$ of Equation\,(\ref{definingeq}) gives a recursive representation of the
first explicit tower attaining the Drinfeld--Vl\u{a}du\c{t} bound, see \cite{inv}. This particular case was
our inspiration to consider Equation\,(\ref{definingeq}).

\vspace{.2cm} \noindent
For a fixed finite field $\fl$ with non-prime $\ell$, Theorem\,\ref{thm3} may give several towers over $\fl$ with
distinct limits; this comes from two sources:
the chosen representation $\ell=q^n$ with $n\geq2$ (i.e., the choice of $q$), and
the chosen partition $n=j+k$. For a cardinality $\ell$ that is neither a prime nor a square, the best lower
bound comes from
\begin{itemize}
\item[] representing $\ell $ as $\ell = p^n$ (i.e., choose $q=p$);
\item[] writing $n \ge 3$ as $n = 2m+1$, choose the partition with $j=m$ and $k=m+1$.
\end{itemize}
The lower bound in Theorem\,\ref{thm3} in this case reads as
$$
\lambda (\mathcal F) \ge  2 \, \biggl( \frac{1}{p^m-1}+ \frac{1}{p^{m+1}-1} \biggr)^{-1} = \frac{2(p^{m+1}-1)}{p+1+\epsilon}\, ,
\ \makebox{ with } \ \epsilon =\frac{p-1}{p^m-1} \, .
$$
This is the tower that proves Theorem\,\ref{thm1}, our main result.

\vspace{.2cm}\noindent
Furthermore, we show that the curves in the tower are related to one-dimensional varieties parametrizing certain $\mathbb{F}_q[T]$-Drinfeld modules of characteristic $T-1$ and rank $n\ge 2$ together with some additional varying structure.

\vspace{.2cm} \noindent
This paper is organized as follows. In Section \ref{sec:two} we investigate the \textquoteleft basic function
field\textquoteright  \
of the tower $\mathcal F$. This is defined as $F = \fl (x,y)$ where $x,y$ satisfy Equation\,(\ref{definingeq}). In
particular, the ramification structure of the extensions $F/\fl (x)$ and $F/\fl (y)$ is discussed in detail. Section
\ref{sec:three}
is the core of our paper. Here we study the tower $\mathcal F = (F_1 \subseteq F_2 \subseteq \cdots)$ and prove Theorem\,\ref{thm3}. The principal difficulty is to show that the genus $g(F_i)$ grows \textquoteleft rather slowly\textquoteright \ as $i \to \infty$. Finally, in Section \ref{sec:four} we show that our tower $\mathcal F$ occurs quite naturally when studying Drinfeld modules of rank $n$, thus providing a modular motivation of the tower.

\vspace{0.5cm}
\noindent We hope that this paper will lead to further developments in the theories of explicit towers and of modular towers, and their relations.

\section{The basic function field}\label{sec:two}

First we introduce some notation.

\vspace{.2cm}

\begin{tabular}{l}
$p$ is a prime number, $q$ is a power of $p$, and $\ell=q^n$ with $n\ge 2$,  \\
$\fl$ is the finite field of cardinality $\ell$, and $\flc$ is the algebraic closure of $\fl$.\\
For simplicity we also denote $K=\fl$ or $\flc$.\\
For an integer $a\ge 1$, we set $\tr_a(T):=T+T^q+\cdots+T^{q^{a-1}}\in K[T].$
\end{tabular}

\begin{remark}\label{rem:traces}
\begin{enumerate}
\item[{\rm (i)}] Let $a,b\ge 1$. Then $(\tr_a\circ\tr_b)(T)=(\tr_b\circ\tr_a)(T)$.
\item[{\rm (ii)}]Let $\Omega \supseteq \fq$ be a field. The evaluation map $\tr_a:\Omega \to \Omega$ is $\fq$-linear; its kernel is contained in the subfield $\fqa \cap \Omega \subseteq \Omega$.
\item[{\rm (iii)}] Let $a,b\ge 1$ and $\gcd(a,b)=1$. Then $K(s)=K(\tr_a(s),\tr_b(s))$ for any $s \in \Omega$.
\end{enumerate}
\end{remark}
\begin{proof}
Item (ii) is clear and the proof of (i) is straightforward. Item (iii) follows by induction from the equation $$\tr_c(T)=\tr_r(T)+(\tr_d(T))^{q^r},$$ which holds whenever $c=d+r>d$.
\end{proof}

\vspace{.2cm}
\noindent We also fix a partition of $n$ into relatively prime integers; i.e., we write
\begin{equation}\label{eq:ij}
n=j+k, \ \makebox{ with integers } j,k \ge 1 \ \makebox{ and } \  \gcd(j,k)=1.
\end{equation}
Without loss of generality we can assume that
\begin{equation}\label{eq:relprime}
\makebox{ $p$ does not divide $j$.}
\end{equation}

\noindent In this section we study the function field $F=K(x,y)$, where $x,y$ satisfy the equation
\begin{equation}\label{eq:basiceq}
\tr_j\left(\frac{y}{x^{q^k}}\right)+\tr_k\left(\frac{y^{q^j}}{x}\right)=1.
\end{equation}
This \textquoteleft basic function field\textquoteright \ $F$ is the first step in
the tower $\mathcal F$ that will be considered in Section\,\ref{sec:three}. We abbreviate
\begin{equation}\label{eq:rs}
R:=\frac{y}{x^{q^k}},\  S:=\frac{y^{q^j}}{x} \ \makebox{ and } \ \alpha:=j^{-1} \in \fp.
\end{equation}
\begin{proposition}\label{prop:rs}
There exists a unique element $u\in F$ such that
\begin{equation}\label{eq:rsu}
R=\tr_k(u)+\alpha \  \makebox{ and } \ S=-\tr_j(u).
\end{equation}
Moreover it holds that
 $K(u)=K(R,S)$   and    $F=K(x,y)=K(x,u)=K(u,y).$
\end{proposition}
\begin{proof}
Let $\Omega \supseteq F$ be an algebraically closed field. Choose $u_0 \in \Omega$ such
that $\tr_k(u_0)=R-\alpha$. Set
$$\mathcal M :=\{\mu \in \Omega \, | \, \tr_k(\mu)=0\}\ \makebox{ and } \ u_\mu:=u_0+\mu, \makebox{ for  } \mu \in \mathcal M.$$
Then $\tr_k(u_\mu)=R-\alpha$ for all $\mu \in \mathcal M$, and by Equation\,(\ref{eq:basiceq})
\begin{eqnarray*}
1 & = & \tr_j(R)+\tr_k(S)=\tr_j(\tr_k(u_\mu)+\alpha)+\tr_k(S)\\
  & = & \tr_k(\tr_j(u_\mu))+j\alpha+\tr_k(S)=\tr_k(S+\tr_j(u_\mu))+1.
\end{eqnarray*}
Hence $S+\tr_j(u_\mu)\in \mathcal M$. The map  $\mu \mapsto S+\tr_j(u_\mu)$ from $\mathcal M$ to itself is injective. To see this, assume that $S+\tr_j(u_0+\mu)=S+\tr_j(u_0+\mu')$ with $\mu,\mu' \in \mathcal M$. Then $\tr_j(\mu-\mu')=0=\tr_k(\mu-\mu')$, hence $\mu-\mu' \in \fqj \cap \fqk=\fq$ and $0=\tr_j(\mu-\mu')=j(\mu-\mu')$. As $j$ is relatively prime to $p$, it follows that $\mu=\mu'$.

\vspace{.2cm} \noindent
Since $\mathcal M$ is a finite set and $0 \in \mathcal M$, there exists some $\mu_0\in \mathcal M$ such that $S+\tr_j(u_{\mu_0})=0$, and then the element $u:=u_{\mu_0}$ satisfies  Equation\,(\ref{eq:rsu}). From item (iii) of Remark\,\ref{rem:traces}, we conclude that $K(u)=K(R,S) \subseteq F$. In particular, the element $u$ belongs to $F$.

\vspace{.2cm} \noindent
To prove uniqueness, assume that $\tilde{u} \in \Omega$ is another element which satisfies Equation\,(\ref{eq:rsu}). Then $\tr_k(\tilde{u})=\tr_k(u)$ and $\tr_j(\tilde{u})=\tr_j(u)$; hence $\tilde{u}-u \in \fqk\cap \fqj=\fq$ and $0=\tr_j(\tilde{u}-u)=j(\tilde{u}-u)$. This implies $\tilde{u}=u$.

\vspace{.2cm} \noindent
The inclusion $K(x,u) \subseteq K(x,y)$ is clear. Conversely, we have $y=Rx^{q^k} \in K(x,u)$ by
Equation\,(\ref{eq:rsu}), hence $K(x,y) \subseteq K(x,u)$. The equality $K(x,y) = K(u,y)$ is shown similarly.
\end{proof}

\begin{proposition}\label{prop:cyclic}
The extension $F/K(u)$ is a cyclic extension of degree $[F:K(u)]=q^n-1$. The elements $x$ and $y$ are Kummer generators for $F/K(u)$, and they satisfy the equations
$$x^{q^n-1}=\frac{-\tr_j(u)}{(\tr_k(u)+\alpha)^{q^j}} \makebox{ and } y^{q^n-1}=\frac{-(\tr_j(u))^{q^k}}{\tr_k(u)+\alpha}.$$
The field $K$ is the full constant field of $F$; i.e., $K$
 is algebraically closed in $F$.\end{proposition}
\begin{proof}
By Equation\,(\ref{eq:rs}),
$$Sx=y^{q^j}=(Rx^{q^k})^{q^j}=R^{q^j}x^{q^n}.$$
Hence, using Equation\,(\ref{eq:rsu}), we obtain
$$x^{q^n-1}=\frac{S}{R^{q^j}}=\frac{-\tr_j(u)}{(\tr_k(u)+\alpha)^{q^j}}.$$
The equation for $y^{q^n-1}$ is proved in the same way. The element $$\frac{-\tr_j(u)}{(\tr_k(u)+\alpha)^{q^j}} \in K(u)$$ has a simple zero at $u=0$; this place is therefore totally ramified in $F=K(x,u)$ over $K(u)$, with ramification index $e=q^n-1$. Hence $[F:K(u)]=q^n-1$, and $K$ is algebraically closed in $F$. As the field $K$ contains all $(q^n-1)$-th roots of unity, the extension $F/K(u)$ is cyclic.
\end{proof}

\begin{corollary}\label{cor:cyclic}
Set $w:=-x^{q^n-1}$ and $z:=-y^{q^n-1}$. Then one has a diagram of subfields of $F$ as in Figure\,{\rm \ref{fig:one}}. The extensions $K(x)/K(w)$, $K(y)/K(z)$ and $F/K(u)$ are all cyclic of degree $q^n-1$; the extensions $F/K(x)$, $F/K(y)$, $K(u)/K(w)$ and $K(u)/K(z)$ are all of degree $q^{n-1}$.
\end{corollary}
\begin{proof}
This follows directly from Proposition\,\ref{prop:cyclic}, since
\begin{equation}\label{eq:zw}
w=-x^{q^n-1}=\frac{\tr_j(u)}{(\tr_k(u)+\alpha)^{q^j}} \ \makebox{ and } \ z=-y^{q^n-1}=\frac{(\tr_j(u))^{q^k}}{\tr_k(u)+\alpha}.
\end{equation}
\end{proof}

\newpage

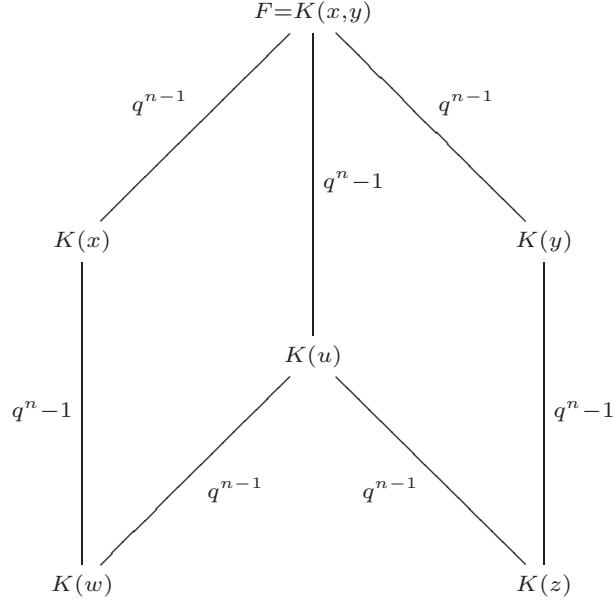
\begin{figure}[h]
\begin{center}
\scalebox{1.2}{
\makebox{
\def\objectstyle{\scriptstyle}
\xymatrix@!=1pc{
&&F=K(x,y) \ar@{-}^{q^{n-1}}[ddrr] \ar@{-}_{q^{n-1}}[ddll] \ar@{-}^{q^n-1}[ddd] && \\
&&&&\\
K(x) \ar@{-}_{q^n-1}[ddd] &&&&  K(y) \ar@{-}^{q^n-1}[ddd] \\
&&K(u) \ar@{-}_{q^{n-1}}[ddrr] \ar@{-}^{q^{n-1}}[ddll]& \\
&&&&\\
K(w) &&&& K(z)\\
}
}
}
\end{center}
\caption{Subextensions of $F$ and their degrees\label{fig:one}}
\end{figure}

\noindent
Our next goal is to describe ramification and splitting of places in the various field extensions in
Figure\,\ref{fig:one}. Denote by $\fl^{\times}$ the multiplicative group of $\fl$. We need some more notation:
\begin{enumerate}
\item[{\rm(i)}] Let $E$ be a function field over $K$ and $0 \neq t \in E$. Then the divisors $$\di(t), \ \di_0(t) \makebox{ and } \di_\infty(t)$$ are the principal divisor, zero divisor and pole divisor of $t$ in $E$. Similarly, the divisor of a nonzero differential $\omega$ of $E$ is denoted by $\di(\omega)$.
\item[{\rm (ii)}] Let $K(t)$ be a rational function field over $K$. Then the places $$[t=\infty] \ \makebox{ and } \  [t=\beta]$$ are the pole of $t$ and the zero of $t-\beta$ in $K(t)$, for any $\beta \in K$.
\item[{\rm (iii)}]Let $E/H$ be a finite separable extension of function fields over $K$. Let $P$ be a place of $H$ and $P'$ a place of $E$ lying above $P$. Then $e(P'|P)$ is the ramification index, and $d(P'|P)$ is the different exponent of $P'$ over $P$.
\end{enumerate}

\begin{proposition}\label{prop:splitting}
For all $\beta \in \fl^{\times}$, the place $[x=\beta]$ splits completely in $F$; i.e., there are $q^{n-1}$ distinct places of $F$ above $[x=\beta]$, all of degree one. For all places $P$ of $F$ above $[x=\beta]$, the restriction of $P$ to $K(y)$ is a place $[y=\beta']$ with some $\beta' \in \fl^{\times}$.
\end{proposition}
\begin{proof}
Upon multiplication by $x^{q^{k-1}}$, Equation\,(\ref{eq:basiceq}) is the minimal polynomial of $y$ over $K(x)$. Substituting $x=\beta$ into this equation we obtain $$\tr_n\left(\frac{y}{\beta^{q^k}}\right)=1,$$ which has $q^{n-1}$ simple roots, all belonging to $\fl^{\times}$.
\end{proof}

\vspace{.2cm}
\noindent
Next we describe ramification in subextensions of $F$. As ramification indices and different exponents do not change under separable constant field extensions, we will assume until the end of Section\,\ref{sec:two} that $K=\flc$. Hence all places of $F$ will have degree one. Recall that $\alpha \in \fp$ and $j\alpha=1$. The following sets will be important:
\begin{equation}\label{eq:gamma set}
\Gamma:=\left\{\gamma \in K \, | \,  \tr_j(\gamma)=0\right\}
\end{equation}
and
\begin{equation}\label{eq:delta set}
\Delta:=\left\{\delta \in K \, | \,  \tr_k(\delta)+\alpha=0\right\}.
\end{equation}
Clearly $\#\Gamma=q^{j-1}$, $\#\Delta=q^{k-1}$ and $\Gamma \cap \Delta = \emptyset$.

\begin{proposition}\label{prop:ramification}
Ramification in $F/K(u)$ is as follows:
\begin{enumerate}
\item[{\rm (i)}] The places $[u=\gamma]$ with $\gamma \in \Gamma$ and $[u=\delta]$ with $\delta \in \Delta$ are totally ramified in $F/K(u)$. We denote by $P_\gamma$ (resp.~$Q_\delta$) the unique place of $F$ lying above $[u=\gamma]$ (resp. above $[u=\delta]$).
\item[{\rm (ii)}] There are exactly $q-1$ places of $F$ above $[u=\infty]$; we denote them by $V_1,\dots,V_{q-1}$. Their ramification indices are $e(V_i|[u=\infty])=(q^n-1)/(q-1)$.
\item[{\rm (iii)}] All other places of $K(u)$ are unramified in $F$.
\end{enumerate}
\end{proposition}
\begin{proof}
Follows from Hasse's theory of Kummer extensions, see \cite[Proposition 3.7.3]{stich}
\end{proof}

\begin{corollary}\label{cor:genus F}
The genus of $F$ is $$g(F)=\frac12 \left((q^n-2)(q^{j-1}+q^{k-1}-2)+(q^n-q)\right).$$
\end{corollary}
\begin{proof}
Apply Hurwitz' genus formula \cite[Theorem 3.4.13]{stich} to the extension $F/K(u)$. Observe that all ramifications in this extension are tame.
\end{proof}

\vspace{.2cm}
\noindent
The next proposition will play an essential role in Section\,\ref{sec:three}. For abbreviation we set
\begin{equation}\label{eq:Nr}
N_r:=\frac{q^r-1}{q-1}, \makebox{ for every integer } r \ge 1.
\end{equation}

\begin{proposition}\label{prop:ram and diff}
Ramification indices and different exponents of the places $P_\gamma$, $Q_\delta$ and $V_i$ in the various subextensions of $F$ are as shown in Figures {\rm \ref{fig:two}, \ref{fig:three}} and {\rm \ref{fig:four}}. All other places in these subextensions are unramified.
\end{proposition}

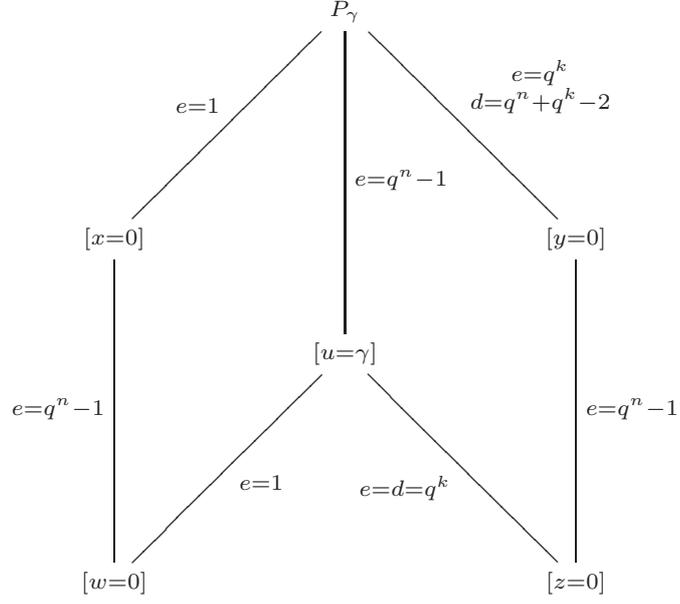
\begin{figure}
\begin{center}
\scalebox{1.2}{
\makebox{
\def\objectstyle{\scriptstyle}
\xymatrix@!=1pc{
&&P_\gamma \ar@{-}^{\substack{e=q^k\\d=q^n+q^k-2}}[ddrr] \ar@{-}_{e=1}[ddll] \ar@{-}^{e=q^n-1}[ddd] && \\
&&&&\\
[x=0] \ar@{-}_{e=q^n-1}[ddd] &&&&  [y=0] \ar@{-}^{e=q^n-1}[ddd] \\
&&[u=\gamma] \ar@{-}_{e=d=q^k}[ddrr] \ar@{-}^{e=1}[ddll]& \\
&&&&\\
[w=0] &&&& [z=0]\\
}
}
}
\end{center}
\caption{Ramification and different exponents for $P_\gamma, \ \gamma \in \Gamma$.\label{fig:two}}
\end{figure}

\begin{figure}
\begin{center}
\scalebox{1.2}{
\makebox{
\def\objectstyle{\scriptstyle}
\xymatrix@!=1pc{
&&Q_\delta \ar@{-}^{e=1}[ddrr] \ar@{-}_{\substack{e=q^j\\d=q^n+q^j-2}}[ddll] \ar@{-}^{e=q^n-1}[ddd] && \\
&&&&\\
[x=\infty] \ar@{-}_{e=q^n-1}[ddd] &&&&  [y=\infty] \ar@{-}^{e=q^n-1}[ddd] \\
&&[u=\delta] \ar@{-}^{e=d=q^j}[ddll] \ar@{-}_{e=1}[ddrr]& \\
&&&&\\
[w=\infty] &&&& [z=\infty]\\
}
}
}
\end{center}
\caption{Ramification and different exponents for $Q_\delta, \ \delta \in \Delta$.\label{fig:three}}
\end{figure}
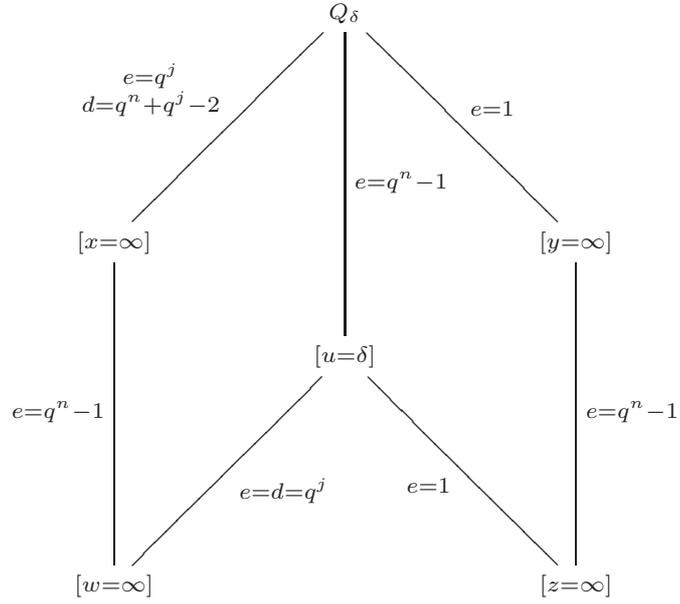

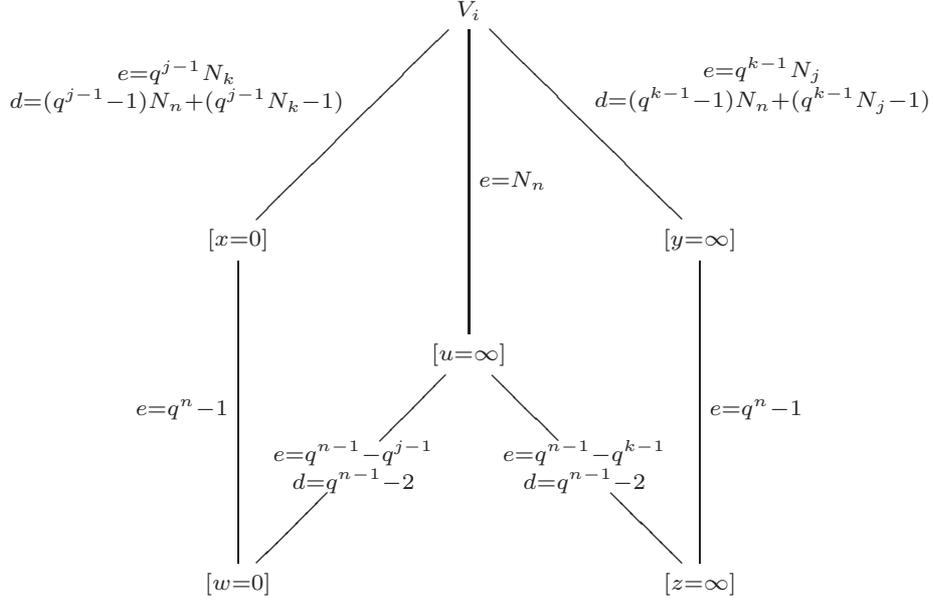
\begin{figure}
\begin{center}
\scalebox{1.2}{
\makebox{
\def\objectstyle{\scriptstyle}
\xymatrix@!=1pc{
&&V_i \ar@{-}^{\substack{e=q^{k-1}N_j\\d=(q^{k-1}-1)N_n+(q^{k-1}N_j-1)}}[ddrr] \ar@{-}_{\substack{e=q^{j-1}N_k\\d=(q^{j-1}-1)N_n+(q^{j-1}N_k-1)}}[ddll] \ar@{-}^{e=N_n}[ddd] && \\
&&&&\\
[x=0] \ar@{-}_{e=q^n-1}[ddd] &&&&  [y=\infty] \ar@{-}^{e=q^n-1}[ddd] \\
&&[u=\infty] \ar@{-}|{\substack{e=q^{n-1}-q^{j-1}\\d=q^{n-1}-2}}[ddll] \ar@{-}|{\substack{e=q^{n-1}-q^{k-1}\\d=q^{n-1}-2}}[ddrr]& \\
&&&&\\
[w=0] &&&& [z=\infty]\\
}
}
}
\end{center}
\caption{Ramification and different exponents for $V_i, \  1\leq i \leq q$.\label{fig:four}}
\end{figure}

\noindent
\begin{proof}
We work out the behaviour of the places $V_i$ in Figure \ref{fig:four}; the other cases are done in a similar way. So we consider a place $V=V_i$ of $F$ lying above the place $[u=\infty]$. It follows from Equation\,(\ref{eq:zw}) that $V$ is a zero of $x$ and $w$, and a pole of $y$ and $z$. Hence the restrictions of $V$ to the subfields $K(x)$, $K(w)$, $K(y)$ and $K(z)$ are the places $[x=0]$, $[w=0]$, $[y=\infty]$ and $[z=\infty]$.

\vspace{.2cm} \noindent
Next we investigate ramification of $[u=\infty]$ over $[w=0]$. From Equation\,(\ref{eq:zw}), the zero and pole divisor of $w$ in $K(u)$ are as follows:
\begin{equation}\label{eq:zero divisor w}
\di_0(w) =\sum_{\gamma \in \Gamma}\,[u=\gamma]+(q^{n-1}-q^{j-1})\,[u=\infty]
\end{equation}
and
\begin{equation}\label{eq:pole divisor w}
\di_\infty(w) =q^j\sum_{\delta \in \Delta}\,[u=\delta].
\end{equation}
Equation\,(\ref{eq:zero divisor w}) shows that $e([u=\infty]|[w=0])=q^{n-1}-q^{j-1}=q^{j-1}(q^k-1).$ The divisor of the differential $dw$ in $K(u)$ is
\begin{equation}\label{eq:divisor dw}
\di(dw) =-2 \, \di_\infty(w)+\diff(K(u)/K(w)),
\end{equation}
where $\diff(K(u)/K(w))$ is the different of the extension $K(u)/K(w)$, see \cite[p.178, Equation\,(4.37)]{stich}. Differentiating the equation
$$\tr_j(u)=(\tr_k(u)+\alpha)^{q^j}\cdot w$$
gives
$$du=(\tr_k(u)+\alpha)^{q^j}\cdot dw$$
and hence
\begin{eqnarray}
\di(dw) & = & \di(du)-q^j\cdot \di(\tr_k(u)+\alpha) \notag\\
        & = & -2\,[u=\infty]-q^j\sum_{\delta \in \Delta}\,[u=\delta]+q^j\cdot q^{k-1}\,[u=\infty]\notag \\
        & = & (q^{n-1}-2)\,[u=\infty]-q^j\sum_{\delta \in \Delta}\,[u=\delta].\label{eq:diveq dw2}
\end{eqnarray}
We substitute (\ref{eq:diveq dw2}) and (\ref{eq:pole divisor w}) into Equation\,(\ref{eq:divisor dw}) and obtain
$$\diff(K(u)/K(w))=(q^{n-1}-2)\,[u=\infty]+q^j\sum_{\delta \in \Delta}\,[u=\delta],$$
hence the different exponent of the place $[u=\infty]$ over $[w=0]$ is
$$d([u=\infty]|[w=0])=q^{n-1}-2.$$
The place extensions $V|[u=\infty]$ and $[x=0]|[w=0]$ are tamely ramified, with ramification indices given by $e(V|[u=\infty])=N_n$ and $e([x=0]|[w=0])=q^n-1$.
We see easily that $$e(V|[x=0])=q^{j-1}N_k,$$ and transitivity of different exponents gives (see \cite[Corollary\,3.4.12]{stich}) $$d(V|[x=0])=(q^{j-1}-1)N_n+(q^{j-1}N_k-1).$$  We have thus proved the left hand side of Figure\,\ref{fig:four}. The proof of the right hand side is analogous.
\end{proof}

\vspace{.2cm}
\noindent
We will also need the following lemma.
\begin{lemma}\label{lem:uzw}
We have $K(u)=K(z,w)$.
\end{lemma}
\begin{proof}
The field $L:=K(z,w)$  is clearly contained in $K(u)$ (see Figure\,\ref{fig:one}). As $F =L(x,y)$ and $x^{q^n-1}, y^{q^n-1} \in L$, it follows that $[F:L]$  divides $(q^n-1)^2$, and therefore $[K(u):L]$ is relatively prime to $p$. But $[K(u):L]$ divides the degree $[K(u):K(w)]=q^{n-1}$, hence $[K(u):L]=1$.
\end{proof}

\vspace{.2cm}
\noindent
For the convenience of the reader, we state Abhyankar's lemma and Hensel's lemma; they will be used frequently in Section \ref{sec:three}.

\begin{proposition}[Abhyankar's lemma] {\rm \cite[Theorem 3.9.1]{stich} } \label{prop:abhyankar}
Let $H$ be a field with a discrete valuation  $\nu :H\to \z\cup\{\infty\}$ having a perfect residue class field. Let $H'/H$ be a finite separable field extension of $H$ and suppose that $H'=H_1 \cdot H_2$ is the composite of two intermediate fields $H\subseteq H_1,H_2 \subseteq H'$. Let $\nu'$ be an extension of $\nu$ to $H'$ and $\nu_i$ the restriction of $\nu'$ to $H_i$, for $i=1,2$. Assume that at least one of the extensions $\nu_1|\nu$ or $\nu_2|\nu$ is tame (i.e., the ramification index $e(\nu_i|\nu)$ is relatively prime to the characteristic of the residue class field of $\nu$). Then one has $e(\nu'|\nu)=\lcm\{e(\nu_1|\nu),e(\nu_2|\nu)\},$ where $\lcm$ means the least common multiple.
\end{proposition}

\begin{proposition}[Hensel's lemma] {\rm \cite[p. 230]{jac}} \label{prop:hensel}
Let $H$ be a field which is complete with respect to a discrete valuation $\nu:H\to \z\cup\{\infty\}$. Let $\mathcal O$ be the valuation ring of $\nu$ and $\mathfrak{m}$ its maximal ideal. Denote by $H^*=\mathcal{O}/\mathfrak{m}$ the residue class field of $\nu$ and by $a \mapsto a^*$ the canonical homomorpism of $\mathcal O$ onto $H^*$. Suppose that the polynomial $\varphi(T)\in \mathcal O[T]$ has the following property:
its reduction $\varphi^*(T)\in H^*[T]$ factorizes as $\varphi^*(T)=\eta_1(T)\cdot \eta_2(T)$ with
$$\eta_1(T),\eta_2(T) \in H^*[T], \,
 \gcd(\eta_1(T),\eta_2(T))=1, \makebox{ and } \eta_1(T) \makebox{ is monic.}$$
Then there are polynomials $\varphi_1(T), \varphi_2(T) \in \mathcal O[T]$ such that $\varphi(T)=\varphi_1(T)\cdot\varphi_2(T)$ with
 $$  \varphi_1(T) \makebox{ is monic, }
 \deg \varphi_1(T)=\deg \eta_1(T),  \ \varphi_1^*(T)=\eta_1(T) \makebox{ and } \varphi_2^*(T)=\eta_2(T).
$$
\end{proposition}

\section{The tower}\label{sec:three}

We keep all notation as before. In this section  we consider a sequence of function fields, \[\mathcal F=(F_1 \subseteq F_2 \subseteq F_3 \subseteq \dots),\]
where $F_1=K(x_1)$ is the rational function field, and for all $i\ge 1$, $F_{i+1}=F_i(x_{i+1})$ with
\begin{equation}\label{eq:defining eq F}
\tr_j\left(\frac{x_{i+1}}{x_{i}^{q^k}}\right)+\tr_k\left(\frac{x_{i+1}^{q^j}}{x_{i}}\right)=1\, .
\end{equation}
 A convenient way to investigate such a sequence is to consider the corresponding `pyramid' of field extensions as shown in Figure \ref{fig:five}. Note that the fields $K(x_i,x_{i+1})$ are isomorphic to the `basic function field' $F=K(x,y)$ that was studied in Section\,\ref{sec:two}.

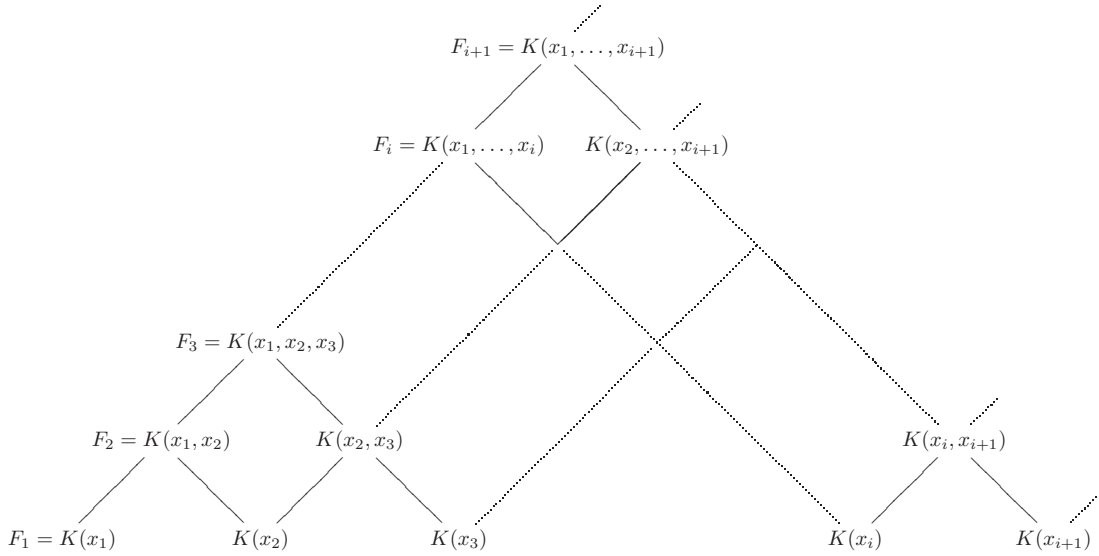
\begin{figure}[h!]
\begin{center}
\resizebox{\textwidth}{!}{
\makebox[\width][c]{
\xymatrix@!=2.5pc@dr{
F_{i+1}=K(x_1,\ldots, x_{i+1}) \ar@{-}[r] \ar@{.}[];[]+/ur 1cm/ & K(x_2,\ldots, x_{i+1})  \ar@{.}[rrr] \ar@{-}[];[d]+0 \ar@{.}[];[]+/ur 1cm/ &  \ar@{.}[]+0;[d]+0 \ar@{.}[rr]&&K(x_i,x_{i+1})\ar@{-}[r]\ar@{.}[];[]+/ur 1cm/& K(x_{i+1}) \ar@{.}[];[]+/ur 1cm/ \\
F_{i}=K(x_1,\ldots, x_{i}) \ar@{-}[];[r]+0 \ar@{-}[u]&\ar@{-}[u] \ar@{.}[rrr]& && K(x_i) \ar@{-}[u]\\
&&&&\\
F_3=K(x_1,x_2,x_3) \ar@{-}[r] \ar@{.}[uu]& K(x_2,x_3) \ar@{.}[uu] \ar@{-}[r]&K(x_3) \ar@{.}[uu]\\
F_2=K(x_1,x_2) \ar@{-}[u] \ar@{-}[r] & K(x_2)\ar@{-}[u]\\
F_1=K(x_1) \ar@{-}[u]
}
}
}
\end{center}
\caption{The pyramid corresponding to $\mathcal F$.\label{fig:five}}
\end{figure}

\begin{proposition}\label{prop:F is tower}
The sequence $\mathcal F$ is a tower of function fields over $K$; i.e., for all $i \ge 1$ the following hold:
\begin{enumerate}
\item[{\rm (i)}] $K$ is the full constant field of $F_i$,
\item[{\rm (ii)}] $F_{i+1}/F_i$ is a separable extension of degree $[F_{i+1}:F_i]>1$,
\item[{\rm (iii)}] $g(F_i) \to \infty$ as $i \to \infty$.
\end{enumerate}
\end{proposition}
\begin{proof}
Equation\,(\ref{eq:defining eq F}) is a separable equation for $x_{i+1}$ over $F_i$, hence $F_{i+1}/F_i$ is separable. Let $P$ be a place of $F_{i+1}$ which lies above the place $[x_1=\infty]$ of $F_1$. From Figure \ref{fig:three} we have ramification indices and different exponents in the pyramid as shown in Figure \ref{fig:six}.

\begin{figure}[h]
\begin{center}
\resizebox{\textwidth}{!}{
\makebox[\width][c]{
\xymatrix@=3pc@dr{
P \ar@{-}|{e=1}[r]+0 \ar@{-}_{\substack{e=q^j\\d=q^n+q^j-2}}[d]+0 & \ar@{.}[]+0;[rr]+0  \ar@{-}|{e=q^j}[]+0;[d]+0 && \ar@{-}|{e=1}[]+0;[r]+0 & \ar@{-}|{e=1}[]+0;[r] \ar@{-}|{e=q^j}[]+0;[d]&[x_{i+1}=\infty]\\
\ar@{-}|{e=1}[]+0;[r]+0 \ar@{.}[]+0;[dd]+0 &\ar@{.}[]+0;[rr]+0 \ar@{.}[]+0;[dd]+0 && \ar@{-}|{e=1}[]+0;[r]&[x_i=\infty]\\
&&&\\
\ar@{-}|{e=1}[]+0;[r]+0 \ar@{-}_{\substack{e=q^j\\d=q^n+q^j-2}}[]+0;[d]+0& \ar@{-}^{\substack{e=q^j\\d=q^n+q^j-2}}[]+0;[d]& \\
 \ar@{-}|{e=1}[]+0;[r] \ar@{-}_{\substack{e=q^j\\d=q^n+q^j-2}}[]+0;[d] & [x_2=\infty] \\
[x_1=\infty]\\
}
}
}
\end{center}
\caption{Ramification over $[x_1=\infty]$.\label{fig:six}}
\end{figure}
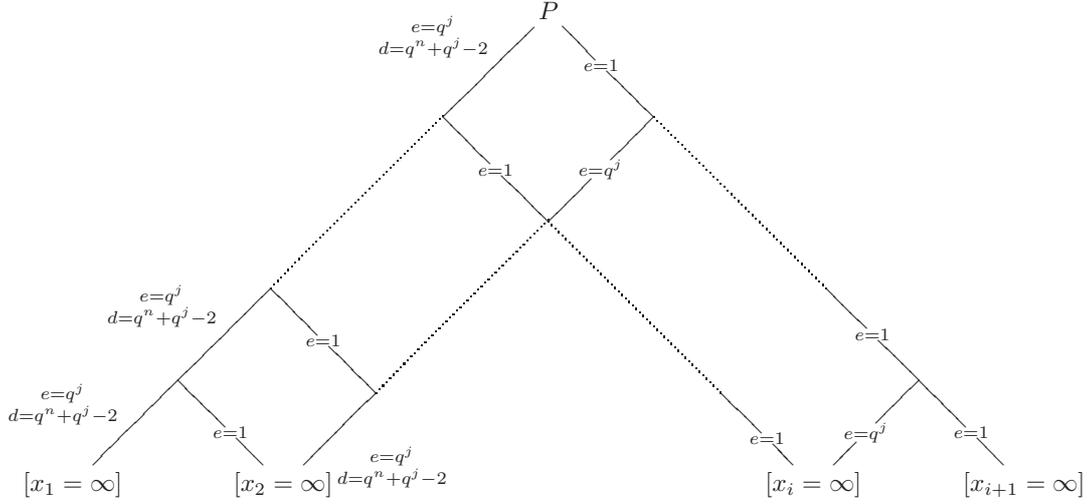

\noindent
As the ramification index of $P$ over $F_i$ is $e=q^j>1$, it follows that $F_i \subsetneqq F_{i+1}$. Let $\beta \in \fl^{\times}$. We show by induction that the rational place $[x_1=\beta]$ splits completely in $F_i/F_1$, for all $i\ge 2$. For $i=2$ this holds by Proposition \ref{prop:splitting}. Assume that the claim holds for $F_i$, and let $P_i$ be a place of $F_i$ lying above $[x_1=\beta]$. Again by Proposition \ref{prop:splitting}, the restriction of $P_i$ to $K(x_i)$ is of the form $[x_i=\beta']$ with some $\beta' \in \fl^{\times}$, and the place $[x_i=\beta']$ splits completely in $K(x_i,x_{i+1})/K(x_i)$. Hence $P_i$ splits completely in $F_{i+1}/F_i$, see \cite[Proposition 3.9.6]{stich}.

\vspace{.2cm} \noindent
We conclude that $F_i$ has places with residue class field $K$, so $K$ is the full constant field of $F_i$. Thus we have shown items (i) and (ii). By Corollary \ref{cor:genus F}, the genus of $F_2=K(x_1,x_2)$ satisfies $g(F_2)\ge 1$.
Since there is some ramified place in every extension $F_{i+1}/F_i$, it follows from  Hurwitz' genus formula that $g(F_i)\to \infty$ as $i\to \infty$.
\end{proof}

\vspace{.1cm}
\noindent
As a consequence of the proof above we get:
\begin{corollary}\label{cor:rational places tower}
All places $[x_1=\beta]$ with $\beta \in \fl^{\times}$ split completely in $F_i/F_1$. In particular in the case $K=\fl$, the number $N(F_i)$ of rational places of $F_i/\fl$ satisfies the inequality\[N(F_i) \ge (\ell-1)\cdot [F_i:F_1].  \]
\end{corollary}

\noindent
In order to prove Theorem \ref{thm3}, one needs an estimate for the genus $g(F_i)$, as $i\to \infty$. Since ramification indices, different exponents and genera are invariant under separable constant field extensions, we will assume from now on that \[K=\flc \makebox{ \ is algebraically closed.}\] Then all places of $F_i/K$ are rational.

\begin{definition}\label{def:b bounded}
{\rm (i)}  Let $E/H$ be a separable extension of function fields over $K$, $P$ a place of $H$ and $b \in \mathbb{R}^{+}$. We say that $P$ is $b$-bounded in $E$ if for any place $P'$ of $E$ above $P$, the different exponent $d(P'|P)$ satisfies \[d(P'|P) \le b\cdot (e(P'|P)-1).\]
{\rm (ii)} Let $\mathcal H =(H_1, H_2, \dots )$ be a tower of function fields over $K$, $P$ a place of $H_1$ and $b \in \mathbb{R}^{+}$. We say that $P$ is $b$-bounded in $\mathcal H$, if it is $b$-bounded in all extensions $H_{i}/H_1$.

\end{definition}

\begin{proposition}\label{prop:b bounded}
Let $\mathcal H =(H_1, H_2, \dots )$ be a tower of function fields over $K$, with $g(H_1)=0$. Assume that $P_1,\dots,P_r$ are places of $H_1$ and $b_1,\dots,b_r \in \mathbb{R}^{+}$ are positive real numbers such that the following hold:
\begin{enumerate}
\item[{\rm (i)}] $P_s$ is $b_s$-bounded in $\mathcal H$, for $1 \le s \le r$.
\item[{\rm (ii)}] All places of $H_1$, except for $P_1,\dots,P_r$, are unramified in $H_i/H_1$.
\end{enumerate}
Then the genus $g(H_i)$ is bounded by
\[g(H_i)-1 \le \left(-1+\frac12 \sum_{s=1}^r b_s\right)\cdot [H_i:H_1].\]
\end{proposition}
\begin{proof}
This is an immediate consequence of Hurwitz' genus formula, see also \cite{recent}.
\end{proof}

\vspace{.2cm}
\noindent
We want to apply Proposition \ref{prop:b bounded} to the tower $\mathcal F$. By Proposition
 \ref{prop:ramification}, only the places $[x_1=0]$ and $[x_1=\infty]$ of $F_1=K(x_1)$ are ramified in $\mathcal F$ (see Figures \ref{fig:two}, \ref{fig:three} and \ref{fig:four}). From Figure \ref{fig:six} and the transitivity of different exponents, the place $[x_1=\infty]$ is $b_{\infty}$-bounded with
\begin{equation}\label{binfty}
b_\infty:=\frac{q^n-1}{q^j-1}+1.
\end{equation}

\noindent \fbox{Main Claim}  
\ \ \ {\it The place $[x_1=0]$ is $b_{0}$-bounded with}
\begin{equation}\label{bzero}
b_0:=\frac{q^n-1}{q^k-1}+1.
\end{equation}

\noindent
Assuming this claim, Proposition\,\ref{prop:b bounded} yields the estimate
\begin{equation}
g(F_i)-1 \le  \left(-1+\frac12 \bigl(b_0+b_\infty \bigr)\right)[F_i:F_1] \
= \ \frac{[F_i:F_1]}{2} \left( \frac{q^n-1}{q^k-1}+\frac{q^n-1}{q^j-1} \right).\label{eq:genus bound}
\end{equation}
For the tower $\mathcal F$ over the field $K=\fl$, we combine Equation\,(\ref{eq:genus bound}) with
Corollary\,\ref{cor:rational places tower} and then we obtain for all $i\ge 2$,
\[\dfrac{N(F_i)}{g(F_i)-1}\ \ge \  2\,\left(\frac{1}{q^j-1}+\frac{1}{q^k-1}\right)^{-1}.\]
Letting $i \to \infty$, this gives a lower bound for the limit $\lambda(\mathcal F)=\lim_{i\to\infty} N(F_i)/g(F_i)$,
\[\lambda(\mathcal F) \ge  2\,\left(\frac{1}{q^j-1}+\frac{1}{q^k-1}\right)^{-1}, \]
and thus proves Theorem\,\ref{thm3}.

\vspace{.2cm}
\noindent
So it remains to prove the Main Claim, which means: for every $i \ge 1$ and every place $\widetilde{P}$ of  $F_{i+1}$ lying above the place $[x_1=0]$, we have to estimate the different exponent $d(\widetilde{P}|[x_1=0])$.

\vspace{.2cm}
\noindent
The restriction of $\widetilde{P}$ to the rational subfield $K(x_{i+1})$ is either the place $[x_{i+1}=0]$ or $[x_{i+1}=\infty]$, as follows from Figures \ref{fig:two}, \ref{fig:three} and \ref{fig:four}. In the case $[x_{i+1}=0]$, the place $\widetilde{P}$ is a zero of all $x_h$ with $1\le h \le i+1$, and we see from
Figure \ref{fig:two} that $\widetilde{P}$ is unramified over $[x_1=0]$; hence we have
$0 = d(\widetilde{P}|[x_1=0]) \le b \cdot (e(\widetilde{P}|[x_1=0])-1)$  for every $b \in \mathbb{R}^{+}$.

\vspace{.2cm} \noindent
The non-trivial case is when $\widetilde{P}$ is a pole of $x_{i+1}$. Then there exists a unique $m \in \{1,\dots,i\}$ such that $\widetilde{P}$ is a zero of $x_m$ and a pole of $x_{m+1}$. The situation is shown in Figure\,\ref{fig:seven}. The question marks indicate that one cannot read off the ramification index and different exponent from ramification data in the basic function field, since both \textquoteleft lower\textquoteright \ ramifications are wild and therefore Abhyankar's lemma does not apply.

\begin{figure}[h]
\begin{center}
\scalebox{1.2}{
\makebox[\width][c]{
\def\objectstyle{\scriptstyle}
\xymatrix@!=2.3pc@dr{
&&&&&\\
\ar@{.}[r] &\ar@{-}^{?}[r] \ar@{.}[u]& \ar@{-}^{?}[r] \ar@{.}[u]&\ar@{-}|{e=1} [r]\ar@{.}[u]& [x_{m+2}=\infty] \ar@{.}[u]\\
\ar@{.}[r] &\ar@{-}^{?}[r] \ar@{-}^{?}[u]&\ar@{-}^{?}[u] \ar@{-}|{e=q^{k-1}N_j}[r]& [x_{m+1}=\infty] \ar@{-}|{e=q^j}[u]\\
\ar@{.}[r]& \ar@{-}|{e=q^k}[r] \ar@{-}^{?}[u]& [x_m=0] \ar@{-}|{e=q^{j-1}N_k}[u]\\
\ar@{.}[r]& [x_{m-1}=0] \ar@{-}|{e=1}[u] \\
}
}
}
\end{center}
\caption{Ramification of $\tilde P$ over $[x_1=0]$: the non-trivial case.\label{fig:seven}}
\end{figure}
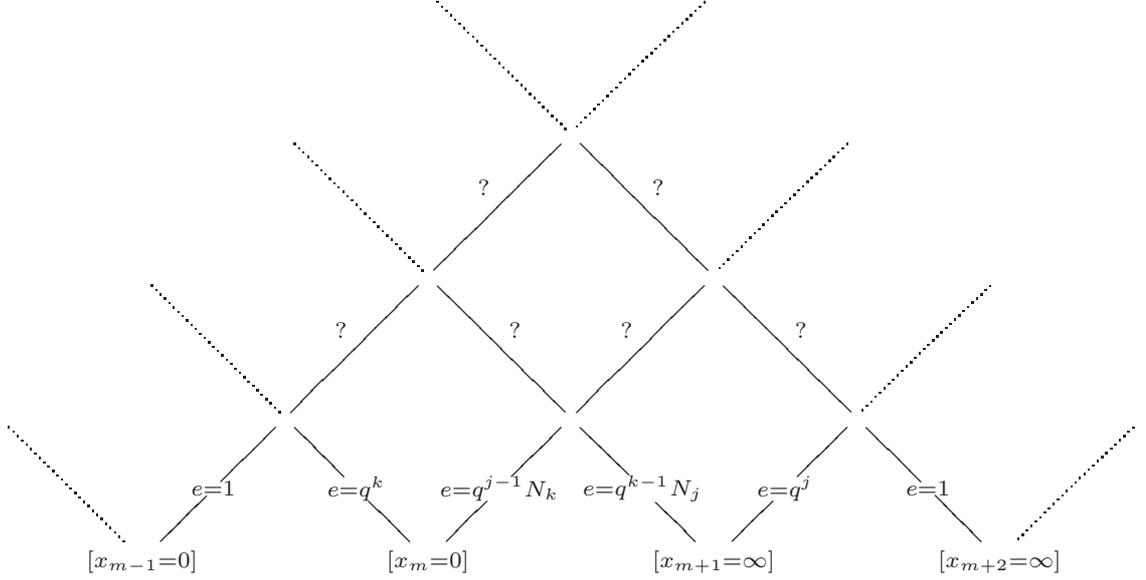

\vspace{.2cm}
\noindent
In order to analyze this situation, we introduce the `$u$-subtower' of $\mathcal F$. Let $u_i \in K(x_i,x_{i+1})$ be the unique element which satisfies the conditions (see
Proposition \ref{prop:rs})
\begin{equation*} \tr_k(u_i)+\alpha=\frac{x_{i+1}}{x_i^{q^k}} \makebox{ \ and \ }
\tr_j(u_i)=-\frac{x_{i+1}^{q^j}}{x_i} \, . \end{equation*}
 We set $z_i:=-x_i^{q^n-1}$ and we then have for all $i\ge 1$ the equations
\begin{equation}\label{eq:subtower E}
z_{i+1}=-x_{i+1}^{q^n-1}=\dfrac{\tr_j(u_{i+1})}{(\tr_k(u_{i+1})+\alpha)^{q^j}}=\dfrac{\tr_j(u_i)^{q^k}}{\tr_k(u_i)+\alpha}\, .
\end{equation}
Equation\,(\ref{eq:subtower E}) defines a subtower $\mathcal E = (E_1 \subseteq E_2 \subseteq \dots)$ of $\mathcal F$ (see Figure\,\ref{fig:eight}), where \[E_i:=K(u_1,u_2,\dots,u_i) \, .\]

\begin{figure}[h]
\begin{center}
\scalebox{1.0}{
\makebox[\width][c]{
\xymatrix@!=2.5pc@dr{
&&&&&\\
F_{4}=E_3(x_1) \ar@{-}[r] \ar@{.}[u]& E_3=K(u_1,u_2, u_3) \ar@{-}[r] \ar@{.}[u]&K(u_2,u_3)\ar@{-}[r]\ar@{.}[u]& K(u_3) \ar@{.}[u] \ar@{-}[r]& \\
F_{3}=E_2(x_1) \ar@{-}[r] \ar@{-}[u]& E_2=K(u_1,u_2) \ar@{-}[u] \ar@{-}[r]&  K(u_2) \ar@{-}[u] \ar@{-}[r] & K(z_3) \ar@{-}[u]\\
F_2=E_1(x_1) \ar@{-}[r] \ar@{-}[u]& E_1=K(u_1) \ar@{-}[u] \ar@{-}[r]&K(z_2) \ar@{-}[u]\\
F_1 \ar@{-}[u]
}
}
}
\end{center}
\caption{The subtower $\mathcal E = (E_1 \subseteq E_2 \subseteq \dots)$.\label{fig:eight}}
\end{figure}
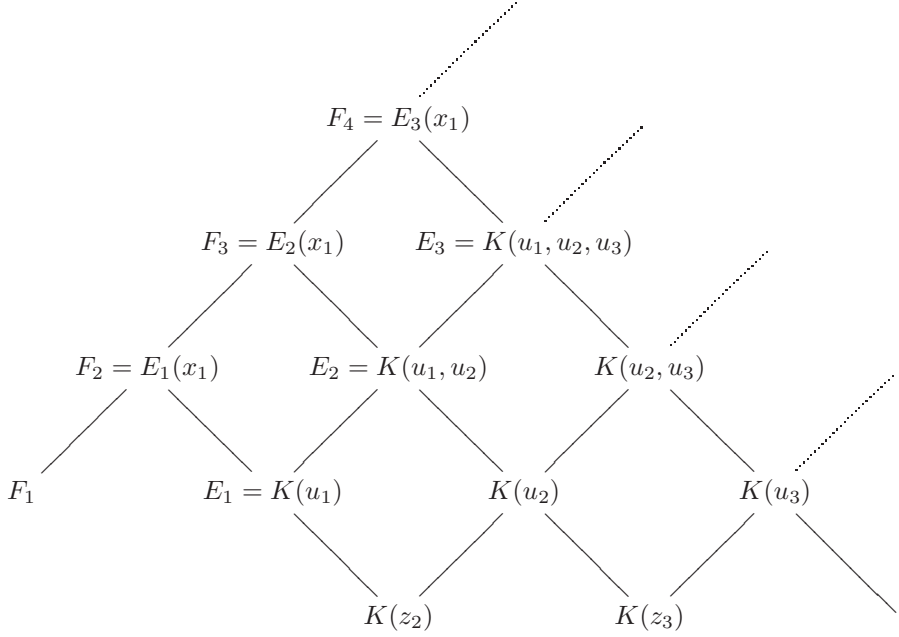

\noindent
By Proposition\,\ref{prop:rs} we know that $F_2=K(x_1,x_2)=K(x_1,u_1)$, and it follows by induction that \[F_{i+1}=E_i(x_1) \makebox{ \ for all \ $i\ge 1$.}\]

\noindent
Let $\widetilde{P}$ be a place of the field $F_{i+1}$ as in Figure\,\ref{fig:seven} (the `non-trivial case'), and let $P$ be the restriction of $\widetilde{P}$ to the subfield $E_i$. The restrictions of $P$ to the subfields $K(u_1),\dots,K(u_i)$ are
\begin{eqnarray}
& &\left[u_s=\gamma_s\right] \makebox{ \ with \ $\gamma_s\in\Gamma$ \ for \ $1\le s \le m-1,$} \notag\\
& &\left[u_m=\infty\right],  \makebox{ and }                                           \label{eq:restriction u} \\
& &\left[u_s=\delta_s\right] \makebox{ \ with\ $\delta_s\in\Delta$ \ for \ $m+1\le s \le i.$} \notag
\end{eqnarray}

\newpage

\noindent
\fbox{The case {\bf $ m=1$.}}

\vspace{.1cm} \noindent We will see that this case already comprises most problems that occur in the general case. The situation is shown in Figure \ref{fig:nine}, where ramifications come from Figures \ref{fig:three} and \ref{fig:four}.

\begin{figure}[h]
\begin{center}
\scalebox{1.3}{
\makebox[\width][c]{
\def\objectstyle{\scriptstyle}
\xymatrix@!@=8pt{
&&&&&\\
[u_1=\infty] \ar@{-}[ur] && [u_2=\delta_2] \ar@{-}[ur] \ar@{-}[ul]&&[u_3=\delta_3] \ar@{.}[ur] \ar@{-}[ul]  \\
&[z_2=\infty] \ar@{-}|{\substack{e=q^{k-1}(q^j-1)\\d=q^{n-1}-2}}[ul] \ar@{-}|{e=d=q^j}[ur] && [z_3=\infty] \ar@{-}|{e=d=q^j}[ur] \ar@{-}|{e=1}[ul]&&[z_4=\infty] \ar@{-}|{e=1}[ul]\\
}
}
}
\end{center}
\caption{The case $m=1$.\label{fig:nine}}
\end{figure}
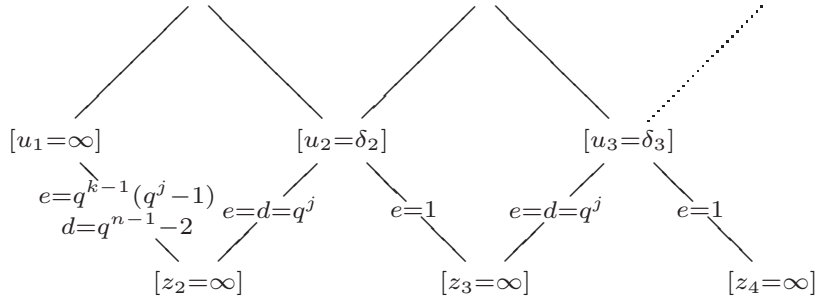

\vspace{.2cm} \noindent
Ramification indices and different exponents do not change under completion; we will therefore  replace the fields $F_s, E_s, K(u_s)$ etc. by their completions $\widehat{F}_s, \widehat{E}_s, \widehat{K(u_s)}$ etc.~(of course, completions are understood at the restrictions of $\widetilde{P}$ to the corresponding fields). As the field $K$ is assumed to be algebraically closed, the ramification indices are then equal to the degrees of the corresponding field extensions.
To simplify notation, we set
\begin{equation}\label{eq:simpler notation}
u:=u_1, \ \  z:=z_2,\ \ H:=\widehat{K(z)}, \makebox{ and } \ E:=\widehat{E}_1=\widehat{K(u)}=H(u).
\end{equation}

\noindent
The next two propositions are of vital importance for the proof of the Main Claim.

\begin{proposition}\label{prop:EtH}
There exists an element $t\in E$ such that
\begin{equation} t^{q^j-1}=z^{-1}. \end{equation}
 The extension $H(t)/H$ is cyclic of degree $[H(t):H]=q^j-1$, and the extension $E/H(t)$ is Galois of degree $[E:H(t)]=q^{k-1}$. The ramification indices and different exponents in the extensions $E \supseteq H(t) \supseteq H$ are as shown in Figure\,{\rm \ref{fig:ten}}.
\end{proposition}

\begin{figure}[h]
\begin{center}
\scalebox{0.9}{
\makebox[\width][c]{
\xymatrix@=4pc{
E=H(u) \ar@{-}^{\quad \txt{$e=q^{k-1}$ \\ $d=2(q^{k-1}-1)=2(e-1)$}}_{\txt{galois of \\ degree $q^{k-1}$}\quad}[d] \\
H(t) \ar@{-}^{\quad \txt{$e=q^{j}-1$\\ $d=q^j-2$}}_{\txt{cyclic of \\ degree $q^j-1$}\quad}[d]\\
H
}
}
}
\end{center}
\caption{The intermediate field $H(t)$.\label{fig:ten}}
\end{figure}

\begin{proof} Notations as in Equation\,(\ref{eq:simpler notation}).
The extension $E/H$ of degree $[E:H]=q^{k-1}(q^j-1)$ is totally ramified, $z^{-1}$ is a prime element of $H$ and $u^{-1}$ is a prime element of $E$. Hence we have \[z^{-1}=\epsilon \cdot \left(u^{-1}\right)^{q^{k-1}(q^j-1)}\]
with a unit $\epsilon \in E$. As an easy consequence of Hensel's lemma, we can write $\epsilon$ as
\[\epsilon=\epsilon_0^{q^j-1} \makebox{\ with a unit \ $\epsilon_0 \in E$.}\]
Then the element
$t:=\epsilon_0 \cdot \left(u^{-1}\right)^{q^{k-1}}$
satisfies the equation
$t^{q^j-1}=z^{-1}.$

\vspace{.2cm} \noindent
It is clear that $H(t)/H$ is a cyclic extension of degree $q^j-1$, and hence the degree of the field
extension $E/H(t)$ is  $[E:H(t)]=q^{k-1}$.

\vspace{.2cm} \noindent
Next we show that  $E/H(t)$ is Galois. From Equation\,(\ref{eq:subtower E}) follows that  $u$ is a root of
the polynomial \[\varphi(T):=z^{-1}\cdot\tr_j(T)^{q^k}-\left(\tr_k(T)+\alpha\right) \in H[T].\]
Its reduction $\varphi^*(T)$ modulo the valuation ideal of $H$ is the polynomial \[\varphi^*(T)=-\left(\tr_k(T)+\alpha\right) \in K[T].\]
We set $\eta_1(T):=\tr_k(T)+\alpha$ and $\eta_2(T):=-1$, then Hensel's lemma gives a factorization
\[\varphi(T)=\varphi_1(T)\cdot \varphi_2(T) \makebox{\ with \ } \varphi_1(T), \varphi_2(T) \in H[T],\]
$\varphi_1(T)$ is monic of degree $q^{k-1}$ and with reduction $\varphi_1^*(T)=\tr_k(T)+\alpha$. Again by Hensel's lemma, the polynomial $\varphi_1(T)$ splits into linear factors over $H$. As $u \not\in H$ is a root of $\varphi(T)$, it follows that $\varphi_2(u)=0$. The degree of the field extension $E=H(u)$ over $H$ is
\[ [E:H]=q^{n-1}-q^{k-1}=\deg \varphi_2(T), \]
and therefore the monic polynomial $z\cdot \varphi_2(T) \in H[T]$ is the minimal polynomial of $u$ over $H$.

\vspace{.2cm} \noindent
We can construct some other roots of $\varphi_2(T)$ in $E$ as follows. By Hensel's lemma, the polynomial
\[\psi(T):=z^{-1}\cdot \tr_j(T)^{q^k}-\tr_k(T) \in H[T]\]
has $q^{k-1}$ distinct roots $\Theta \in H$. For any such $\Theta$ we have
\begin{equation*}
\varphi(u+\Theta)  =  z^{-1}\cdot\tr_j(u+\Theta)^{q^k}-\left(\tr_k(u+\Theta)+\alpha\right)
  =  \varphi(u)+\psi(\Theta)=0.
\end{equation*}
Since $u+\Theta \not\in H$, we conclude that $u+\Theta$ is a root of $\varphi_2(T)$; hence we obtain an automorphism of the field $E$ over $H$ by setting $u \mapsto u+\Theta.$
For $\Theta \neq 0$, this automorphism has order $p={\rm char}(K)$. As $[H(t):H]=q^{j}-1$ is relatively prime to $p$, the restriction of this automorphism to $H(t)$ is the identity. We have thus constructed $q^{k-1}$ distinct automorphisms of $E$ over $H(t)$. This proves that the extension $E/H(t)$ is Galois, since its degree is $[E:H(t)]=q^{k-1}$.

\vspace{.2cm}
\noindent
The different exponent of $E/H$ is $q^{n-1}-2$, see Figure \ref{fig:nine}. Since $H(t)/H$  is tamely ramified with different exponent $q^j-2$, one obtains easily that $E/H(t)$ has different exponent $2(q^{k-1}-1)$, by  transitivity of the different.
\end{proof}

\vspace{.2cm}
\noindent
Note that we are still considering the case $m=1$ with completions at the corresponding places. We define now subfields $G_s \subseteq \widehat{E}_s$  (see Figure \ref{fig:eleven}) by setting:
\[G_1:=H(t), \ \makebox{ and } \ G_{s+1}:=G_s(u_{s+1}) \makebox{\ for \ } s \ge 1.\]

\begin{figure}[h]
\begin{center}
\resizebox{\textwidth}{!}{
\makebox[\width][c]{
\def\objectstyle{\scriptstyle}
\xymatrix@!=1.45pc@dr{
&&&&\\
\widehat{E_3}\ar@{-}[r] \ar@{.}[u]&G_3\ar@{-}|{e=q^j-1}[r] \ar@{.}[u]& \ar@{-}|{e=1}[rr] \ar@{.}[u]&&\widehat{K(u_3)} \ar@{.}[u]\\
&&&&\\
\widehat{E_2} \ar@{-}[r] \ar@{-}[uu]& G_2 \ar@{-}[uu] \ar@{-}|{e=q^j-1}[r]&\widehat{K(u_2)} \ar@{-}|{e=d=q^j}[uu] \ar@{-}|{e=1}[rr]&&\widehat{K(z_3)} \ar@{-}|{e=d=q^j}[uu]\\
&&\\
\widehat{E_1}=\widehat{K(u_1)} \ar@{-}[r] \ar@{-}[uu]& G_1=H(t) \ar@{-}[uu] \ar@{-}|{e=q^j-1}[r]&H=\widehat{K(z_2)} \ar@{-}|{e=d=q^j}[uu]
}
}
}
\end{center}
\caption{The subfields $G_s$.\label{fig:eleven}}
\end{figure}
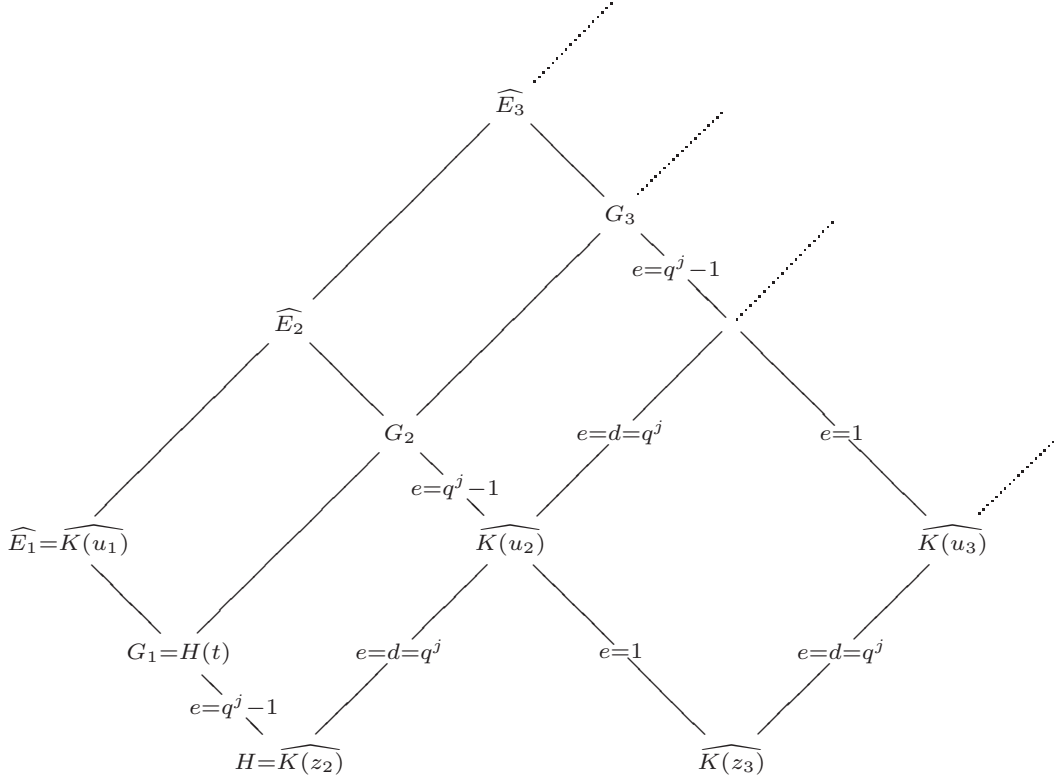

\noindent
Ramification indices and different exponents in Figure\,\ref{fig:eleven} can be read off from
Figures \ref{fig:nine} and \ref{fig:ten}. From this it follows in particular that $\widehat{K(u_2)}=\widehat{K(z_3)}$, $\widehat{K(u_2,u_3)}=\widehat{K(u_3)}=\widehat{K(z_4)}$, etc.

\begin{proposition}\label{prop:G s+1/G s}
For all $s\ge 1$, the extension $G_{s+1}/G_s$ is Galois of degree $q^j$. The ramification index of $G_{s+1}/G_s$ is $e=q^j$, and the different exponent is $d=2(q^j-1)=2(e-1)$.
\end{proposition}
\begin{proof}
It is clear that $[G_{s+1}:G_s]=q^j$. From transitivity of the different, the different exponent of $G_{s+1}/G_s$ is $d=2(q^j-1)$. It remains to prove that the extension $G_{s+1}/G_s$ is Galois.
For simplicity we write $v:=u_{s+1}$, $w:=z_{s+1}$ and $G:=G_s$. Then $G_{s+1}=G(v)$.

\vspace{.2cm} \noindent Denote by $\mathcal O_G$ and $\mathfrak m_G$ the valuation ring of $G$ and its maximal ideal. By Equation (\ref{eq:subtower E}), the element $v$ is a root of the polynomial
\[\Phi(T):=\left(\tr_k(T)+\alpha\right)^{q^j}-w^{-1}\cdot \tr_j(T) \in \mathcal O_G[T].\]
The reduction of $\Phi(T)$ modulo $\mathfrak m_G$ decomposes in $K[T]$ as follows:
\begin{equation}\label{eq:Phistar}
\Phi^*(T):=\left(\tr_k(T)+\alpha\right)^{q^j}= \prod_{\delta \in \Delta} \eta_{\delta}(T) \ \makebox{ with } \
\eta_{\delta}(T) = T^{q^j} - \delta^{q^j} \ \in K[T].
\end{equation}

\vspace{.2cm} \noindent
The polynomials
$\eta_{\delta}(T)$
are relatively prime, for distinct $\delta \in \Delta$. By Hensel's lemma we can lift the decomposition of $\Phi^*(T)$ to a decomposition of $\Phi(T)$ as follows:
\[\Phi(T)=\prod_{\delta \in \Delta}\Phi_\delta(T) , \]
with
monic polynomials\  $ \Phi_\delta(T) \in \mathcal O_G[T]$
of degree $\deg \Phi_\delta(T)=q^j$\,, and
\begin{equation}\label{eq:redphidelta}
\Phi_\delta^*(T)=T^{q^j}-\delta^{q^j} \makebox{\  for all \ } \delta \in \Delta.
\end{equation}
As $v$ is a root of $\Phi(T)$ and $[G(v):G]=q^j$, we conclude that there is a unique $\epsilon \in \Delta$ such that
\[\Phi_\epsilon(T) \makebox{ is the minimal polynomial of $v$ over $G$.}\]
We will now show that the polynomial $\Phi_\epsilon(T)$ has $q^j$ distinct roots in $G_{s+1}$ and hence that the field extension $G_{s+1}/G_s$ is Galois. To this end we consider
\[\chi(T):=\tr_j(T)-w\cdot \tr_k(T)^{q^j} \in G[T].\]
From Figure \ref{fig:eleven} we see that $w$ has a pole of order $q^j-1$ in $G$, and hence we can write
\[w=\left(\frac{1}{w_0}\right)^{q^j-1} \makebox{ with some prime element $w_0 \in G$.}\]
Then
\begin{equation*}
\chi(T)  =  \tr_j(T)-\left(\frac{1}{w_0}\right)^{q^j-1}\tr_k(T)^{q^j}
  =  w_0\left[\frac{T}{w_0}-\left(\frac{T}{w_0}\right)^{q^j}+w_0\cdot \Lambda\left(\frac{T}{w_0}\right)\right],
\end{equation*}
with $\Lambda(Z)$ a polynomial in the ring $\mathcal O_G[Z]$.
Again by Hensel's lemma, there exist $q^j$ distinct elements $\xi \in \mathcal O_G$ such that
\[\xi-\xi^{q^j}+w_0\cdot \Lambda(\xi)=0,\]
and for these elements we have
$\chi(w_0\xi)=0.$
Now it follows that
\begin{equation*}
\Phi(v+w_0\xi)  =  (\tr_k(v+w_0\xi)+\alpha)^{q^j}-w^{-1}\tr_j(v+w_0\xi)
  =  \Phi(v)-w^{-1}\chi(w_0\xi)=0-0=0.
\end{equation*}
For every $\delta \in \Delta\backslash \{\epsilon\}$ we have
$\Phi_\delta^*(v+w_0\xi)^*=\Phi_\delta^*(v^*)=\epsilon^{q^j}-\delta^{q^j},$
and therefore $v+w_0\xi$ cannot be a root of the polynomial $\Phi_\delta(T)$. Hence the element $v+w_0\xi$ is a root of $\Phi_\epsilon(T)$, which is the minimal polynomial of $v$ over $G$.
\end{proof}

\vspace{.2cm}
\noindent
In the following we need the concept of weakly ramified extensions of valuations. For simplicity, we  consider
only the case of complete fields.
\begin{definition}\label{def:weakly ramified}
Let $L$ be a field which is complete with respect to a discrete valuation and has an algebraically closed residue class field of characteristic $p>0$. A finite separable extension $L'/L$ is said to be weakly ramified if the following hold:
\begin{enumerate}
\item[{\rm (i)}] There exists a chain of intermediate fields \[L=L_0 \subseteq L_1 \subseteq L_2 \dots \subseteq L_m=L'\] such that all extensions $L_{i+1}/L_i $ are Galois $p$-extensions.
\item[{\rm (ii)}] The different exponent $d(L'|L)$ satisfies \[d(L'|L)=2(e(L'|L)-1),\] where $e(L'|L)$ denotes the ramification index of $L'/L$.
\end{enumerate}
\end{definition}

\begin{proposition}\label{prop:weakly ramified}
Let $L$ be a field, complete with respect to a discrete valuation, with an algebraically closed residue  field of characteristic $p>0$, and let $L'/L$ be a finite separable extension.
\begin{enumerate}
\item[{\rm (i)}] Let $H$ be an intermediate field, $L\subseteq H \subseteq L'$. Then $L'/L$ is weakly ramified if and only if both extensions $H/L$ and $L'/H$ are weakly ramified.
\item[{\rm (ii)}] Assume that $L'=H_1\cdot H_2$ is the composite field of two intermediate fields $L \subseteq H_1, H_2 \subseteq L'$. If both extensions $H_1/L$ and $H_2/L$ are weakly ramified, then also $L'/L$ is weakly ramified.
\end{enumerate}
\end{proposition}
\begin{proof}
This follows using techniques from \cite{recent}.
See \cite{recent}.
\end{proof}

\vspace{.2cm}
\noindent
By Propositions \ref{prop:EtH}, \ref{prop:G s+1/G s} and item\,(i) of Proposition\,\ref{prop:weakly ramified},
the extensions $\widehat{E}_1/G_1$ and $G_s/G_1$ are weakly ramified. We conclude from item (ii) of Proposition \ref{prop:weakly ramified} that
\begin{equation}\label{eq:weakly ramified}
\widehat{E}_s/\widehat{E}_1 \makebox{ is weakly ramified, for all $s\ge 1$.}
\end{equation}

\noindent
Now we can calculate the different exponent of a place $\widetilde{P}$ of $F_{i+1}$ over $P_1:=[x_1=0]$ in the case $m=1$ (see Figures \ref{fig:eight} and \ref{fig:eleven}). As before, the place $P$ is the restriction of $\widetilde{P}$ to the field $E_i$, and we denote by $P_2$ the restriction of $\widetilde{P}$ to the field $F_2=E_1(x_1)$. The situation is represented in Figure \ref{fig:twelve}, where we set $e_0:=e(P_2|P_1)$ and $e_1:=e(P|[u_1=\infty])$. By Equation (\ref{eq:weakly ramified}),  $e_1$ is a power of $p$, and
$d(P|[u_1=\infty])=2(e_1-1).$

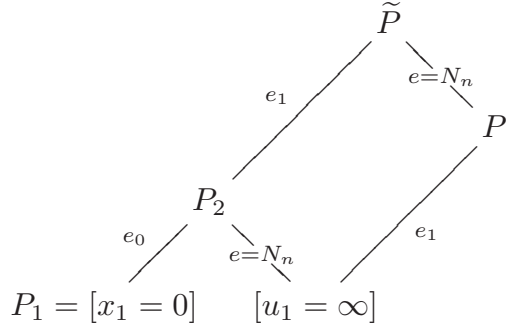
\begin{figure}[h]
\begin{center}
\scalebox{1.2}{
\makebox[\width][c]{
\xymatrix@dr{
\widetilde{P} \ar@{-}|{e=N_n}[r] \ar@{-}_{e_1}[dd]&P \ar@{-}^{e_1}[dd]\\
&\\
P_2 \ar@{-}|{e=N_n}[r] \ar@{-}_{e_0}[d] &[u_1=\infty]\\
P_1=[x_1=0]
}
}
}
\end{center}
\caption{The place extension $\widetilde{P}|P_1$ in the case $m=1$.\label{fig:twelve}}
\end{figure}

\noindent By Proposition\,\ref{prop:ramification} the place $P_2$ is tame over $[u_1=\infty]$ with ramification index \[e(P_2|[u_1=\infty])=\frac{q^n-1}{q-1}.\]
From transitivity we obtain
\begin{equation*}
d(\widetilde{P}|[u_1=\infty])  =  e_1\left(\frac{q^n-1}{q-1}-1\right)+d(\widetilde{P}|P_2)
  =  \frac{q^n-1}{q-1}\cdot 2(e_1-1)+\left(\frac{q^n-1}{q-1}-1\right);
\end{equation*}
hence
\begin{equation}\label{eq:different ptilde ptwo}
d(\widetilde{P}|P_2)=\left(\frac{q^n-1}{q-1}+1\right)(e_1-1)=(N_n+1)(e_1-1).
\end{equation}
By Figure \ref{fig:four},
\begin{equation}\label{eq:e0 and d(p2 p1)}
e_0=e(P_2|P_1)=q^{j-1}N_k \makebox{\ and \ } d(P_2|P_1)=(q^{j-1}-1)N_n+(e_0-1).
\end{equation}
Combining Equations (\ref{eq:different ptilde ptwo}) and (\ref{eq:e0 and d(p2 p1)}) one gets
\begin{eqnarray*}
d(\widetilde{P}|P_1) &  =  & e_1\left( (q^{j-1}-1)N_n+(e_0-1) \right)+ (N_n+1)(e_1-1)\\
\\
                     &  =  & N_n(e_1q^{j-1}-1)+(e_0e_1-1) \
                      =  \ \left(\frac{N_n}{N_k}+1\right)(e_0e_1-1)+\frac{N_n}{N_k}-N_n\\
\\
                     & \le & \left(\frac{q^n-1}{q^k-1}+1\right)(e(\tilde{P}|P_1)-1).
\end{eqnarray*}
This inequality shows the Main Claim in case $m=1$; i.e., the place extension $\tilde{P}|P_1$ satisfies
$$ d(\tilde{P}|P_1) \le  \left(\frac{q^n-1}{q^k-1}+1\right) (e(\tilde{P}|P_1)-1) = b_0\cdot( e(\tilde{P}|P_1)-1).$$
It remains to prove the Main Claim  for:

\vspace{.2cm} \noindent
\fbox{The case $m\ge 2$.}

\noindent
Now we have a place $\widetilde{P}$ of the field $F_{i+1}$ such that its restriction $P$ to the field $E_i=K(u_1,\dots,u_i)$ satisfies the condition in Equation (\ref{eq:restriction u}) for some integer $m$, with $2 \le m \le i$. The restrictions of $P$ to the rational subfields $K(u_1),\dots,K(u_m)$ and $K(z_2),\dots,K(z_m)$ are shown in Figure \ref{fig:thirteen}.

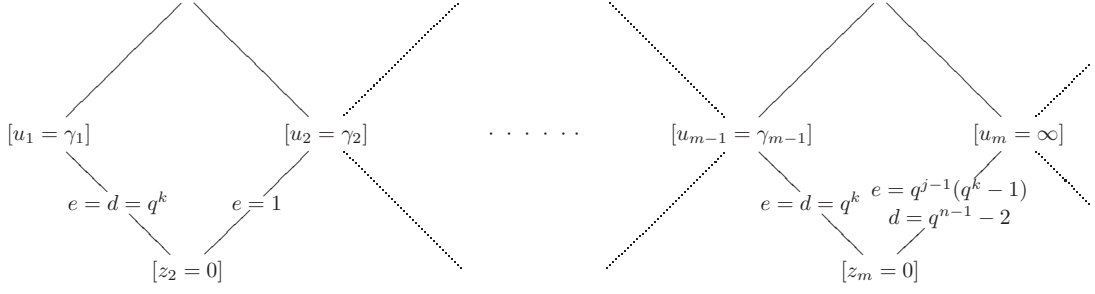
\begin{figure}[h]
\begin{center}
\resizebox{\textwidth}{!}{
\makebox[\width][c]{
\renewcommand{\labelstyle}{\textstyle}
\xymatrix@!=3.3pc{
&&&&&&&\\
[u_1=\gamma_1] \ar@{-}[ur] && [u_2=\gamma_2] \ar@{.}[ur] \ar@{-}[ul]  \ar@{.}[dr] &\ar@{}|{\cdot \ \cdot \ \cdot \ \cdot \ \cdot \ \cdot }[r]&&[u_{m-1}=\gamma_{m-1}] \ar@{-}[ur] \ar@{.}[ul]  \ar@{.}[dl] &&[u_m=\infty] \ar@{.}[];[]+/ur 1.6cm/ \ar@{-}[ul] \ar@{.}[];[]+/dr 1.6cm/\\
&[z_2=0] \ar@{-}|{e=d=q^k}[ul] \ar@{-}|{e=1}[ur] &&&&& [z_m=0] \ar@{-}|{e=d=q^k}[ul] \ar@{-}|{\txt{$e=q^{j-1}(q^k-1)$ \\ $d=q^{n-1}-2$}}[ur]\\
}
}
}
\end{center}
\caption{The case $m \ge 2$.\label{fig:thirteen}}
\end{figure}

\vspace{.2cm}
\noindent
There is a strong analogy between Figures \ref{fig:nine} and \ref{fig:thirteen} that interchanges the roles of $j$ and $k$. In fact, after passing to the completions one proves that there is a field $L_1$ with $\widehat{K(z_m)} \subseteq L_1 \subseteq \widehat{K(u_m)}$ such that:
\begin{enumerate}
\item[(i)] $L_1/\widehat{K(z_m)}$ is cyclic with ramification index $e=q^k-1$.
\item[(ii)] $\widehat{K(u_m)}/L_1$ is a weakly ramified $p$-extension.
\item[(iii)] The extension $L:=\widehat{E}_{m-1}\cdot L_1$ is weakly ramified over $L_1$.
\end{enumerate}
The proof is exactly as in the case of $m=1$; we leave the details to the reader. Note that $L_1$ corresponds to the field $G_1$ in Figure \ref{fig:eleven}. From the case of $m= 1$, we  know that the extension $\widehat{G}/\widehat{K(u_m)}$ with $G=K(u_m,\dots,u_i)$ is weakly ramified (see Equation (\ref{eq:weakly ramified})), and then it follows from Proposition \ref{prop:weakly ramified} that also $\widehat{E}_i=\widehat{E}_m\cdot \widehat{G}$ is weakly ramified over $L$. From item (i) above we see that the extension $L/\widehat{E}_{m-1}$ has ramification index $e=q^k-1$. The extensions $\widehat{E}_{m-1}/\widehat{E}_{1}$ and $\widehat{F}_{2}/\widehat{F}_{1}$ are unramified. Figure \ref{fig:fourteen} represents the situation (in  Figure \ref{fig:fourteen},
\textquoteleft w.r.\textquoteright \  means `weakly ramified'). The degree of $M:=L\cdot\widehat{F}_2$ over $\widehat{F}_2$ follows from Abhyankar's lemma.

\begin{figure}[h]
\begin{center}
\scalebox{1.0}{
\makebox[\width][c]{
\renewcommand{\labelstyle}{\textstyle}
\xymatrix@!=2.5pc@dr{
&&&&&&\\
\widehat{F}_{i+1} \ar@{-}[r]|{e=N_n} & \widehat{E}_i  \ar@{-}[rr]|{w.r.}&& \widehat{G}  \\
 & \widehat{E}_m \ar@{-}[u]|{w.r.} \ar@{-}[rr]|{w.r.}& & \widehat{K(u_m)} \ar@{-}[u]|{w.r.}\\
M=L\cdot\widehat{F}_2 \ar@{-}[uu]|{\widetilde{e} \ , \ \, \widetilde{d}} \ar@{-}|{e=N_n}[r]& L \ar@{-}[u]|{w.r.} \ar@{-}[rr]|{w.r.}& & L_1 \ar@{-}[u]|{w.r.}\\
 & \widehat{E}_{m-1} \ar@{-}[u]|{e=q^k-1}\\
\widehat{F}_2 \ar@{-}[uu]|{e=N_k} \ar@{-}[r]|{e=q^n-1} & \widehat{E}_1\ar@{-}[u]|{e=1}\\
\widehat{F}_1 \ar@{-}[u]|{e=1}
}
}
}
\end{center}
\caption{Ramification over $\widehat{F}_1$ in case $m\ge 2$.\label{fig:fourteen}}
\end{figure}
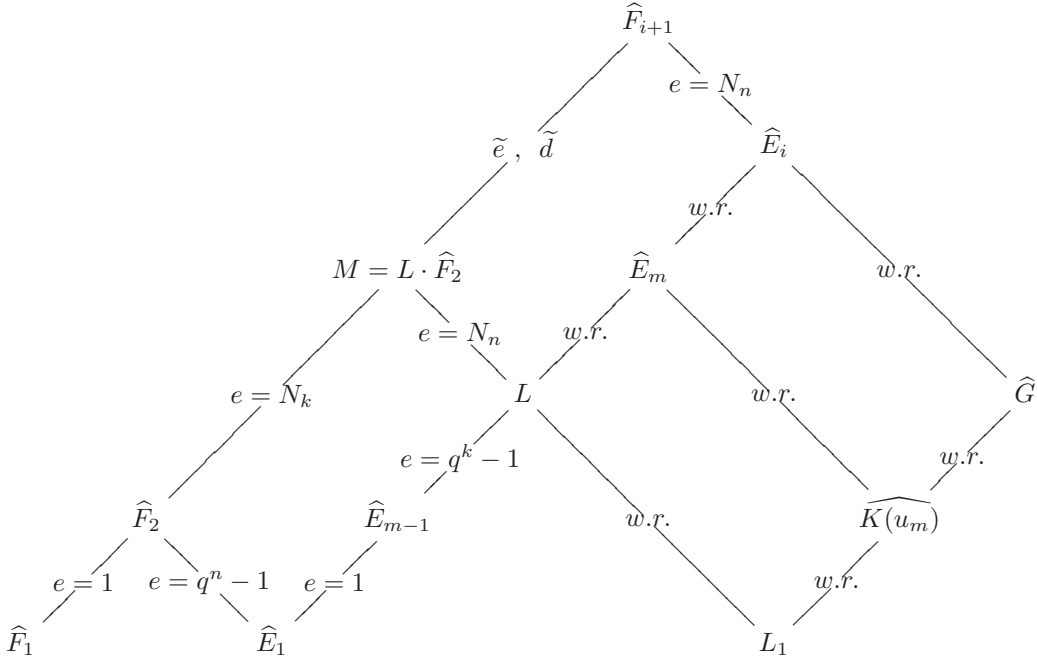

\vspace{.2cm}
\noindent
Finally we consider the composite field $\widehat{F}_{i+1}=\widehat{E}_i\cdot \widehat{F}_2=\widehat{E}_i(x_1)$ and determine ramification index and different exponent of $\widetilde{P}$ over $P_1=[x_1=0]$. We have
\[e(\widetilde{P}|P_1)=\dfrac{q^k-1}{q-1}\cdot \, \widetilde{e}\ , \ \makebox{ with $\tilde{e}$ a power of $p$}.\]

\noindent Denoting by $\widetilde{d}$ the different of $\widehat{F}_{i+1}$ over $M$,
\[\widetilde{d}+\widetilde{e}\cdot(N_n-1)=(N_n-1)+N_n\cdot(2\widetilde{e}-2), \, \makebox{ and
hence } \, \widetilde{d}=(N_n+1)(\widetilde{e}-1).\]
 We finally obtain that
\begin{equation*}
d(\widetilde{P}|P_1)   = \widetilde{d}+\widetilde{e}\cdot(N_k-1)=  \left( \dfrac{q^n-1}{q^k-1}+1
\right)(e(\widetilde{P}|P_1)-1)-\left(\dfrac{q^n-1}{q-1}-\dfrac{q^n-1}{q^k-1}\right)
  \le  b_0\cdot(e(\widetilde{P}|P_1)-1).
\end{equation*}
This finishes the proof of
the Main Claim  and hence also the proof of Theorem \ref{thm3}. \rule{0.5em}{0.5em}

\begin{remark}\label{separated variables}
{\rm  The $u$-tower $\mathcal{E} = (E_1 \subseteq E_2 \subseteq \ldots)$ is recursively defined by
(see Equation\,(\ref{eq:subtower E}))
\begin{equation}\label{eq:sepvaru}
\frac{\tr_j(Y)}{(\tr_k(Y)+\alpha)^{q^j}}=
\frac{\tr_j(X)^{q^k}}{\tr_k(X)+\alpha} \, .
\end{equation}
This equation has \textquoteleft separated variables\textquoteright .

\vspace{.2cm}
\noindent
The $x$-tower $\mathcal{F} = (F_1 \subseteq F_2 \subseteq \ldots )$ is recursively defined by Equation\,(\ref{eq:defining eq F})
which does not have separated variables. Subtracting Equation\,(\ref{eq:defining eq F}) from its $q$-th power,
one sees that the tower $\mathcal{F}$ also satisfies the recursive equation
\begin{equation}\label{eq:sepvarx}
\frac{Y^{q^n}-Y}{Y^{q^j}} = \frac{X^{q^n}-X}{X^{q^n-q^k+1}}\, .
\end{equation}
Equation\,(\ref{eq:sepvarx}) has separated variables but it is not irreducible. One can get from this equation a very simple
proof of Corollary \ref{cor:rational places tower}.

\vspace{.2cm}
\noindent
The $z$-tower $\mathcal{H} = (H_1 \subseteq H_2 \subseteq \cdots)$ with $z_i$ as in Equation\,(\ref{eq:subtower E})
and $H_i:=K(z_1, \ldots, z_i)$ satisfies the recursion
\begin{equation}\label{eq:sepvarz}
\frac{(Y+1)^{N_n}}{Y^{N_j}} = \frac{(X+1)^{N_n}}{X^{q^kN_j}}\, ,
\end{equation}
which has separated variables and is reducible. Equation\,(\ref{eq:sepvarz}) can be deduced  from Equation\,(\ref{eq:sepvarx}).

\vspace{.2cm}
\noindent
Since $\mathcal{H}$ is a subtower of $\mathcal{E}$ and $\mathcal{E}$ is a subtower of $\mathcal{F}$,
we have (see \cite{jnt})
\begin{equation}\label{eq:hef}
\lambda (\mathcal{H}) \ge \lambda(\mathcal{E}) \ge \lambda(\mathcal{F})\ge   2\,\left(\frac{1}{q^j-1}+\frac{1}{q^k-1}\right)^{-1} \, .
\end{equation}
It would be interesting to have a direct proof for the limit $\lambda(\mathcal{H})$ just using
Equation\,(\ref{eq:sepvarz}).
}
\end{remark}

\section{A Drinfeld modular motivation for the equations}\label{sec:four}

In this section we show that Equations (\ref{eq:sepvarx}) and (\ref{eq:sepvarz}) can be obtained using a Drinfeld modular construction. 
More precisely, we show that curves in the towers are related to one-dimensional varieties parametrizing certain classes of Drinfeld modules of characteristic $T-1$ and rank $n\ge 2$ together with some varying additional structure. For definitions and results about Drinfeld modules, we refer to \cite{goss}. We will restrict to the case of Drinfeld $\mathbb F_q[T]$-modules of rank $n$ and characteristic $T-1$.

\vspace{.2cm}
\noindent
It is well known, that Drinfeld modular curves which parametrize Drinfeld modules of rank two together with some level structure, have many $\mathbb{F}_{q^2}$-rational points after suitable reductions. These rational points correspond to supersingular Drinfeld modules. In fact, it was shown in \cite{gekeler2}, that curves obtained in this way are asymptotically optimal; i.e., they attain the Drinfeld--Vl\u{a}du\c{t} bound. More generally, after reduction, the variety parametrizing Drinfeld modules of rank $n$ (again with some additional structure) has points corresponding to supersingular Drinfeld modules, which were shown to be $\mathbb{F}_{q^n}$-rational in \cite{gekeler}. This variety however, has dimension $n-1$. Hence we will consider one-dimensional subvarieties containing the supersingular points, to obtain curves with many $\mathbb{F}_{q^n}$-rational points. We will consider a one-dimensional sub-locus corresponding to particular Drinfeld modules (including all supersingular ones, in order to get many rational points), together with particular isogenies leaving this sub-locus invariant (in order to get recursive equations).

\vspace{.2cm}
\noindent
More precisely, let $A=\mathbb F_q[T]$ be the polynomial ring over $\mathbb F_q$. Let $L$ be a field containing $\mathbb F_q$ together with a fixed $\mathbb{F}_q$-algebra homomorphism $\iota: A \to L$. The kernel of $\iota$ is called the characteristic of $L$. We will always assume that the characteristic is the ideal generated by $T-1$. Further denote by $\tau$ the $q$-Frobenius map and let $L\{\tau\}$ be the ring of additive polynomials over $L$ under operations of addition and composition (also called twisted polynomial ring or Ore ring). Given $f(\tau)=a_0+a_1\tau+\cdots+a_n \tau^n \in L\{\tau\}$, we define $D(f):=a_0$. Note that the map $D: L\{\tau\} \to L$ is a homomorphism of $\mathbb{F}_q$-algebras.

\vspace{.2cm}
\noindent
A homomorphism of $\mathbb{F}_q$-algebras $\phi:A \to L\{\tau\}$ (where one usually writes $\phi_a$ for the image of $a\in A$ under $\phi$) is called a Drinfeld $A$-module $\phi$ of characteristic $T-1$ over $L$, if $D\circ \phi=\iota$ and if there exists $a \in A$ such that $\phi_a \neq \iota(a)$. It is determined by the additive polynomial $\phi_T$. If $\phi_T=g_0\tau^n+g_1\tau^{n-1}+\cdots+g_{n-1}\tau+1 \in L\{\tau\}$, with $g_i\in L$ and $g_0\neq 0$, the Drinfeld module is said to have rank $n$. A Drinfeld module $\phi$ given by $\phi_T=g_0\tau^n+g_1\tau^{n-1}+\cdots+g_{n-1}\tau+1$ is called supersingular (in characteristic $T-1$) if $g_1=\cdots=g_{n-1}=0$. Note that this corresponds to the situation that the additive polynomial $\phi_{T-1}$ is purely inseparable.

\vspace{.2cm}
\noindent
For Drinfeld modules $\phi$ and $\psi$ as above, an isogeny $\lambda:\phi\to \psi$ over $L$ is an element $\lambda\in L\{\tau\}$ satisfying
\begin{equation}
\label{isogeny}
\lambda \cdot \phi_a=\psi_a \cdot \lambda \text{ for all } a\in A.
\end{equation}
We say that the kernel of the isogeny is annihilated by multiplication with $P(T)\in \mathbb{F}_q[T]$, if there exists $\mu\in L\{\tau\}$ such that
$$\mu\cdot \lambda=\phi_{P(T)}.$$ Over the algebraically closed field $\bar{L}$, Drinfeld modules $\phi$ and $\psi$ are isomorphic, if they are related by an invertible isogeny; i.e., if there exists $\lambda \in \bar{L}^\times$, such that Equation (\ref{isogeny}) holds.

\vspace{.2cm}
\noindent
In analogy to normalized Drinfeld modules in \cite{elkiesd}, for $1 \le j \le n$, let $\mathfrak D_{n,j}$ be the set of rank $n$ Drinfeld $A$-modules of characteristic $T-1$ of the form $\phi_T=-\tau^n+g\tau^{j} + 1$. We will call such Drinfeld modules normalized. As before, we assume that $\gcd(n,j)=1$ and write $k=n-j$. Note that $\mathfrak D_{n,j}$ contains the supersingular Drinfeld module $\phi$ with $\phi_T=-\tau^n+1$. First we exhibit certain isogenies for $\phi \in \mathfrak D_{n,j}$ and show that the isogenous Drinfeld module is again in $\mathfrak D_{n,j}$:

\begin{proposition}
Let $\phi \in \mathfrak D_{n,j}$ be a Drinfeld module defined by $\phi_T=-\tau^n+g\tau^{j} + 1$ and let $\lambda$ be an additive polynomial of the form $\lambda=\tau^k-a$. Then there exists a Drinfeld module $\psi$ such that $\lambda$ defines an isogeny from $\phi$ to $\psi$ if and only if
\begin{equation}\label{eq:ag}
\dfrac{1}{a}g^{q^k}-\dfrac{1}{a^{q^j}}g-a^{q^n-1}+1=0.
\end{equation}
Moreover $\psi \in \mathfrak D_{n,j}$ and more precisely, $\psi_T=-\tau^n+h\tau^{j} + 1$ with
\begin{equation}\label{eq:hag}
h=-a^{q^n}+a+g^{q^k}.
\end{equation}
\end{proposition}
\begin{proof}
The existence of a Drinfeld module $\psi$ such that $\lambda$ defines an isogeny from $\phi$ to $\psi$, is equivalent to the existence of an additive polynomial $\psi_T=h_0\tau^n+h_1\tau^{n-1}+\cdots+h_{n-1}\tau + h_n$ such that $\lambda \cdot \phi_T=\psi_T \cdot \lambda$. Clearly one needs to choose $h_0=-1$ and $h_n=1$.

The equation $\lambda \cdot \phi=\psi \cdot \lambda$ implies that
$$-\tau^{n+k}+(a+g^{q^k})\tau^n-ag \tau^{j}+\tau^k-a=\sum_{i=k}^{n+k}h_{n-i+k}\tau^i-\sum_{i=0}^nh_{n-i}a^{q^{i}}\tau^i.$$ Consequently we have
\begin{equation}\label{eq:h}
(-a^{q^n}+a+g^{q^k})\tau^n-ag \tau^{j}=\sum_{i=k+1}^{n+k-1}h_{n-i+k}\tau^i-\sum_{i=1}^{n-1}h_{n-i}a^{q^{i}}\tau^i.
\end{equation}
By comparing coefficients of $\tau^{i}$ in Equation\,(\ref{eq:h}) for $n+k-1 \le i \le n+1$, we see that $h_{i}=0$ for all $1\le i < k$. By considering coefficients of $\tau^{i}$ in Equation\,(\ref{eq:h}) for $i \not\equiv n \pmod{k}$, we then conclude that $h_i=0$ for all $i \not\equiv 0 \pmod{k}$. We are left to determine the coefficients of the form $h_{i}$, with $1\le i < n$ a multiple of $k$. Again by Equation\,(\ref{eq:h}) we conclude that for such $i$, $h_i=0$ if $k<i<n$. This leaves two equations in the coefficient $h_k$, namely the ones in Equation\,(\ref{eq:h}) relating the coefficients of  $\tau^n$ and $\tau^j$: $$-a^{q^n}+a+g^{q^k}=h_k \ \makebox{ and } \ -ag=-h_{k}a^{q^j}.$$ The proposition now follows.
\end{proof}

\vspace{.2cm}
\noindent
Now we determine all solutions of Equation\,(\ref{eq:ag}):
\begin{proposition}\label{prop:hgx}
Let $X\in \bar{L}$ be such that $X^{q^k-1}=a$. All solutions of Equation\,{\rm(\ref{eq:ag})} are given by
\begin{equation}
\label{eq:allsol}
g=\dfrac{X^{q^n-1}+c}{X^{q^j-1}}, \ \makebox{ with } \ c\in\mathbb{F}_{q^k}.
\end{equation}
The corresponding $h$ in Equation\,{\rm(\ref{eq:hag})} is given by
\begin{equation}
\label{eq:allsol2}
h=\dfrac{X^{q^n-1}+c}{X^{q^n-q^k}}.
\end{equation}
\end{proposition}
\begin{proof}
Multiplying both sides of Equation\,(\ref{eq:ag}) with $X^{q^n-1}$ and using that $a=X^{q^k-1}$, we find that
$$-X^{(q^n-1)q^k}+X^{q^n-1}+X^{(q^j-1)q^k}g^{q^k}-X^{q^j-1}g=0,$$ which can be rewritten as
$$\left(-X^{q^n-1}+gX^{q^j-1}\right)^{q^k}-\left(-X^{q^n-1}+gX^{q^j-1}\right)=0.$$ The possible solutions for $g$ now follow. Inserting these solutions in Equation\,(\ref{eq:hag}), the formula for $h$ is obtained readily.
\end{proof}

\vspace{.2cm}
\noindent
Note that in fact $X$ corresponds to a choice of a nonzero element in the kernel of the isogeny $\lambda=\tau^k-a$. The kernel of $\lambda$ is just $\mathbb{F}_{q^k}X$. Exactly in the case that $c=-1$, the element $X$ can be chosen to be a $T$-torsion point of the Drinfeld module $\phi$. We will assume from now on that this is the case. The equations relating $g$, $h$ and $X$ then simplify to
\begin{equation}\label{eq:simplified}
g=\dfrac{X^{q^n}-X}{X^{q^j}} \, \makebox{ and } \, h=\dfrac{X^{q^n}-X}{X^{q^n-q^k+1}}.
\end{equation}
Therefore we have obtained exactly the same correspondence as the one described in Equation\,(\ref{eq:sepvarx}).
Through this correspondence, we are parametrizing normalized rank $n$ Drinfeld modules together with an isogeny of the form $\lambda=\tau^k-a$ and a nonzero $T$-torsion point in its kernel.

In the particular case of $n=2$ and $k=1$, these are just normalized Drinfeld modules together with $T$-isogenies (together with a $T$-torsion point in their kernel) as studied by Elkies in \cite{elkiesd}. For general $n$ and $k$, not all of the kernel of $\lambda$ will be annihilated by multiplication with $T$ anymore, but by multiplication with the polynomial $(T-1)^{N_k}-(-1)^k$ (which is obviously relatively prime to the characteristic $T-1$):
\begin{proposition}\label{prop:pk}
Let $\phi$ and $\psi$ be two Drinfeld modules given by $\phi_T=-\tau^n+g\tau^{j} + 1$ and $\psi_T=-\tau^n+h\tau^{j} + 1$. Further let $\lambda=\tau^k-X^{q^k-1}$ be an isogeny from $\phi$ to $\psi$. Then the kernel of $\lambda$ is annihilated by the polynomial $P_k(T)=(T-1)^{N_k}-(-1)^k$.
\end{proposition}
\begin{proof}
Any additive polynomial of the form $\tau-(\alpha X)^{q-1}$ with $\alpha \in \mathbb{F}_{q^k}^{\times}$ is a right factor of $\tau^k-X^{q^k-1}$. In total this gives $N_k$ distinct right factors, since $\#\{ \alpha^{q-1} \, | \, \alpha \in  \mathbb{F}_{q^k}^{\times} \}=N_k$. Clearly the kernel of such a right factor is contained in the kernel of $\tau^k-X^{q^k-1}$. This gives rise to $N_k(q-1)=q^k-1$ nonzero elements of the kernel of $\tau^k-X^{q^k-1}$. Therefore the union of the kernels of the right factors $\tau-(\alpha X)^{q-1}$, with $\alpha \in \mathbb{F}_{q^k}^ {\times}/\mathbb{F}_{q}^{\times}$, is equal to the kernel of $\tau^k-X^{q^k-1}$.

\vspace{.2cm}
\noindent
We claim that the kernel of $\tau-(\alpha X)^{q-1}$ is annihilated by $T-1+\alpha^{q^{j}-1}$. Indeed, $\phi_{T-1+\alpha^{q^{j}-1}}=-\tau^n+g\tau^j+\alpha^{q^{j}-1}$ can be written as  $\mu (\tau-(\alpha X)^{q-1})$ for some additive polynomial $\mu$ if and only if $-(\alpha X)^{q^n-1}+g(\alpha X)^{q^j-1}+\alpha^{q^{j}-1}=0$. This equality is satisfied, as can be seen by using Equation (\ref{eq:simplified}) and the fact that $\alpha \in \mathbb{F}_{q^k}^{\times}$.

Two right factors $\tau-(\alpha X)^{q-1}$ and $\tau-(\alpha' X)^{q-1}$ are equal if and only if $\alpha^{q-1}=\alpha'^{q-1}$. Therefore the proposition follows once we show that the product of $T-1+\alpha^{q^{j}-1}$ over all $\alpha \in \mathbb{F}_{q^k}^{\times}/\mathbb{F}_{q}^{\times}$ equals $(T-1)^{N_k}-(-1)^k$. This is the case, since
$$
\begin{array}{rclcl}
\prod_{\beta \in (\mathbb{F}_{q^k}^{\times})^{q-1}}(T-1+\beta^{N_j}) & = & \prod_{\beta \in (\mathbb{F}_{q^k}^{\times})^{q-1}}(T-1+\beta) & = & (-1)^k\prod_{\beta \in (\mathbb{F}_{q^k}^{\times})^{q-1}}(-T+1-\beta)\\
\\
 & = & (-1)^k\left((-T+1)^{N_k}-1\right) & = & (T-1)^{N_k}-(-1)^k.
 \end{array}
$$
In the first equality we used that since $\gcd(j,k)=1$, the map from $(\mathbb{F}_{q^k}^{\times})^{q-1}$ to itself given by $\beta \mapsto \beta^{N_j}$ is a bijection. In the third equality we used that $(\mathbb{F}_{q^k}^{\times})^{q-1}$ consists of exactly all elements of $\mathbb{F}_{q^k}^{\times}$ of multiplicative order dividing $N_k$.
\end{proof}

\vspace{.2cm}
\noindent
From the proof of Proposition \ref{prop:pk} we also see that $P_k(T)$ is the lowest degree polynomial annihilating the kernel of $\lambda=\tau^k-X^{q^k-1}$. For $k=1$, we have $P_1(T)=T$, so the kernel of the isogeny $\lambda$ is annihilated by multiplication with $T$.

\vspace{.2cm}
\noindent
Alternatively, instead of studying normalized rank $n$ Drinfeld modules, one can consider the corresponding $\bar{L}$-isomorphism classes. More precisely, we look at isomorphism classes of rank $n$ Drinfeld modules $\phi$ with $\phi_T=g_0\tau^n+g_1\tau^{n-1}+\cdots+g_{n-1}\tau+1$ such that
$$g_1=\ldots =g_{j-1}=g_{j+1}=\ldots =g_{n-1}=0.$$
Clearly every such class contains a normalized Drinfeld module, and two normalized Drinfeld modules $\phi, \phi' \in \mathfrak D_{n,j}$, with $\phi_T=-\tau^n+g\tau^{j} + 1$ and $\phi'_T=-\tau^n+g'\tau^{j} + 1$ are isomorphic over $\bar{L}$ if and only if $g'=g\cdot \lambda^{q^j-1}$ for some $\lambda\in \mathbb F_{q^n}^{\times}$.

Since $\gcd(n,j)=1$, the image of the map $\lambda \mapsto \lambda^{q^j-1}$ is $(\mathbb F_{q^n}^{\times})^{q-1}$. We see that $\phi$ and $\phi'$ as above are isomorphic if and only if
$$g'^{N_n}=g^{N_n}.$$
We denote $J(\phi)=g^{N_n}$, since it plays the analogous role of the $j$-invariant for normalized Drinfeld modules (also compare with \cite{potemine}).
It is now easy to relate $J(\phi)$ and $J(\psi)$ for Drinfeld modules $\phi$ and $\psi$ which are related by an isogeny of the form $\tau^k-X^{q^k-1}$. By Equation\,(\ref{eq:simplified}) we have
$$J(\phi)=g^{N_n}=\Bigl(\dfrac{X^{q^n-1}-1}{X^{q^j-1}}\Bigr)^{N_n}=\dfrac{(X^{q^n-1}-1)^{N_n}}{(X^{q^n-1})^{N_j}}$$
and similarly
$$J(\psi)=h^{N_n}=\dfrac{(X^{q^n-1}-1)^{N_n}}{(X^{q^n-1})^{q^kN_j}}.$$
Letting $Z=-X^{q^n-1}$, we have
$$J(\phi)=(-1)^k\dfrac{(Z+1)^{N_n}}{Z^{N_j}},\quad J(\psi)=(-1)^k\dfrac{(Z+1)^{N_n}}{Z^{q^kN_j}},$$
which is the same correspondence as the one described in Equation\,(\ref{eq:sepvarz}).


\bibliographystyle{99}

\noindent Alp Bassa\\
Sabanc{\i} University, MDBF\\
{\rm 34956} Tuzla, \.Istanbul, Turkey\\
bassa@sabanciuniv.edu\\

\noindent Peter Beelen\\
Technical University of Denmark, Department of Applied Mathematics and Computer Science\\
Matematiktorvet, Building 303B\\
DK-2800, Lyngby, Denmark\\
p.beelen@mat.dtu.dk\\

\noindent
Arnaldo Garcia\\
Instituto Nacional de Matem\'atica Pura e Aplicada, IMPA\\
Estrada Dona Castorina 110\\
22460-320, Rio de Janeiro, RJ, Brazil\\
garcia@impa.br\\

\noindent Henning Stichtenoth\\
Sabanc{\i} University, MDBF\\
{\rm 34956} Tuzla, \.Istanbul, Turkey\\
henning@sabanciuniv.edu\\

\end{document}